\documentclass{amsart}
\usepackage{mathtools, microtype, mathrsfs, xfrac, hyperref, amssymb, enumerate, xspace, xypic}
\usepackage[latin1]{inputenc}
\usepackage{tikz-cd}

\usepackage{todonotes}

\numberwithin{subsection}{section}

\allowdisplaybreaks[1]

\newtheorem*{namedtheorem}{\theoremname}
\newcommand{\theoremname}{testing}

\newtheorem{theorem}{Theorem}[section]

 \newtheorem*{untheorem}{Theorem}
\newtheorem{proposition}[theorem]{Proposition}
\newtheorem{proposition-definition}[theorem]
{Proposition-Definition}
\newtheorem{corollary}[theorem]{Corollary}
\newtheorem{lemma}[theorem]{Lemma}

\theoremstyle{definition}

\newtheorem{remark}[theorem]{Remark}

\theoremstyle{remark}

\newcommand{\B}{\mathscr{B}}

\newcommand\GG{\mathbb{G}} 
 
 \newcommand\LL{\mathbb{L}}

\newcommand\bA{\mathbb{A}} 
\newcommand\bC{\mathbb{C}} 
 \newcommand\bF{\mathbb{F}}
\newcommand\bG{\mathbb{G}} \newcommand\bH{\mathbb{H}}

 \newcommand\bN{\mathbb{N}}
 \newcommand\bP{\mathbb{P}}
 \newcommand\bR{\mathbb{R}}

 \newcommand\bZ{\mathbb{Z}}

\newcommand\rK{\mathrm{K}} 
 
\newcommand\rO{\mathrm{O}} 
 
 \newcommand\rT{\mathrm{T}}

\newcommand\arr{\ifinner\to\else\longrightarrow\fi}

\newcommand\arrto{\ifinner\mapsto\else\longmapsto\fi}

\def\displaytimes_#1{\mathrel{\mathop{\times}\limits_{#1}}}

\def\displayotimes_#1{\mathrel{\mathop{\bigotimes}\limits_{#1}}}

\newcommand\aut{\operatorname{Aut}}

\newcommand\id{\mathrm{id}}

\renewcommand{\setminus}{\smallsetminus}

\newcommand{\gm}{\GG_{m}}

\newcommand{\GL}{\mathrm{GL}}
\newcommand{\SL}{\mathrm{SL}}
\newcommand{\PGL}{\mathrm{PGL}}
\newcommand{\ga}{\GG_{a}}

\renewcommand{\Im}{\mathrm{Im}}

\renewcommand{\O}{\rO}

\newcommand{\SO}{\mathrm{SO}}

\newcommand{\Spin}{\mathrm{Spin}}

\newcommand{\Pin}{\mathrm{Pin}}

\newcommand{\kstack}{\rK_{0}(\mathrm\overline{{\rm Stack}_{k}})}

\newcommand{\Spec}{\mathrm{Spec}\,}

\begin{document}

\title[Motivic class of $\B G_2$ and $\B\Spin_n$]{On the motivic class of the classifying stack \\ of $G_2$ and the spin groups}

\author[Roberto Pirisi]{Roberto Pirisi}

\address[Pirisi]{Department of Mathematics\\
The University of British Columbia\\
1984 Mathematics Road\\
Vancouver BC\\
V6T 1Z2 Canada}

\email{rpirisi@math.ubc.ca}

\author[Mattia Talpo]{Mattia Talpo}

\address[Talpo]{Department of Mathematics\\
Simon Fraser University\\
8888 University Drive\\
Burnaby BC\\
V5A 1S6 Canada}

\email{mtalpo@sfu.ca}

\maketitle

\begin{abstract}
We compute the class of the classifying stack of the exceptional algebraic group $G_2$ and of the spin groups $\Spin_7$ and $\Spin_8$ in the Grothendieck ring of stacks, and show that they are equal to the inverse of the class of the corresponding group. Furthermore, we show that the computation of the motivic classes of the stacks $\B\Spin_n$ can be reduced to the computation of the classes of $\B \Delta_n$, where $\Delta_n\subset \Pin_n$ {is the ``extraspecial $2$-group'',} the preimage of the diagonal matrices under the projection $\Pin_n\to \O_n$ to the orthogonal group.
\end{abstract}

\tableofcontents

\section{Introduction}

The Grothendieck ring of algebraic stacks  $\kstack$ over a field $k$ is generated, as an abelian group, by equivalence classes $\{X\}$ of algebraic stacks of finite type over $k$ with affine stabilizers. These classes are subject to the ``scissor relations'' $\{X\}=\{Y\}+\{X\setminus Y\}$ for a closed substack $Y\subseteq X$, and moreover to the relations $\{E\}=\{\bA^n\times_k X\}$ for a vector bundle $E\to X$ of rank $n$. The product is defined on generators by setting $\{X\}\cdot \{Y\}=\{X\times_k Y\}$, and extended by linearity. This object, a variant of the more well-known Grothendieck ring of varieties, was introduced by Ekedahl in \cite{ekedahl-grothendieck}, after it had appeared in different guises in several earlier works \cite{toen-grothendieck-stacks, behrend-dhillon, joyce-motivic-invariant}

Given an affine algebraic group $G$ over $k$, the classifying stack $\B G$ has a class $\{\B G\}$ in $\kstack$, whose computation is an interesting problem, morally related to Noether's problem of stable rationality of fields of invariants for $G$ (see the discussion in \cite[Section 6]{ekedahl-finite-group}).

For a connected group $G$, a first guess for the value of $\{\B G\}$ in $\kstack$ is the inverse of the class of $G$ itself: the formula $\{\B G\}=\{G\}^{-1}$ is true for special groups (i.e. those groups for which every $G$-torsor is locally trivial for the Zariski topology), because in this case for every $G$-torsor $X\to Y$ we have $\{X\}=\{G\}\cdot \{Y\}$. In particular this applies to the universal $G$-torsor $\Spec k\to \B G$.
On the other hand, Ekedahl has shown in \cite{ekedahl-approximation} that if $G$ is a connected and reductive non-special group and the characteristic of $k$ is $0$, then there exists a $G$-torsor $X\to Y$ such that $\{X\}\neq \{G\}\cdot \{Y\}$. In view of this result, it is perhaps more natural to expect that $\{\B G\}\neq \{G\}^{-1}$ for non-special $G$.

The class of $\{\B G\}$ for non-special $G$ has been computed in a few cases: for $G=\PGL_2$ and $\PGL_3$ by Bergh \cite{bergh-motivic-classes}, for $\SO_n$ with odd $n$ by Dhillon and Young \cite{dhillon-young-SOn} and for any $n$ by the second author and Vistoli \cite{talpo-vistoli}. In all these cases we have indeed (somewhat surprisingly) that $\{\B G\}=\{G\}^{-1}$ in $\kstack$.

Assume that $k$ is a field {of characteristic different from $2$, and containing a square root of $-1$}. In this paper we compute the class of $\B G$ in $\kstack$ for the {split forms of the} exceptional group $G_2$ and the spin groups $\Spin_7$ and $\Spin_8$ (note that $\Spin_n$ is special for $2\leq n\leq 6$), and show that in these cases as well, the class of $\B G$ coincides with $\{G\}^{-1}$.

\begin{untheorem}
The equality $\{\B G\}=\{G\}^{-1}$ holds in $\kstack$ for $G=G_2, \Spin_7, \Spin_8$.
\end{untheorem}

We also set up a computation for $\Spin_n$ for a general $n$, reducing it to the computation of the class of the classifying stacks $\B \Delta_n$ for the sequence of finite groups $\Delta_n\subset \Pin_n$ obtained as preimage of the diagonal matrices of $\O_n$ along the projection $\Pin_n\to \O_n$ (Theorem  \ref{thm:main}) {(these are called the ``diagonal'' extraspecial $2$-groups in \cite{wood})}. This allows us to deduce (Corollary \ref{cor:main}) that $\{\B\Spin_n\}=\{\Spin_n\}^{-1}$ holds for every $n\geq 2$ if and only if $\{\B \Delta_n\}=1$ holds for every $n$ (for a finite group $G$, the ``expected class'' of $\{\B G\}$ is $1$, see \cite[Section 3]{ekedahl-finite-group}). 

Our overall approach
is the same as the method used in previous papers on this topic: we consider a linear action of $G$ on a vector space $V$, and stratify the space in locally closed pieces $X_i$ for which the class of the quotient $[X_i/G]$ can be explicitly computed (in some cases we will have to iterate this method). The equality $\{\B G\}\LL^{\dim V}=\{[V/G]\}=\sum_i \{[X_i/G]\}$ (coming from the scissor relations and the fact that $[V/G]\to \B G$ is a vector bundle of rank $n$) then allows us to also compute the class of $\B G$.

We believe it to be the case that for {$n\geq 15$} we have $\{\B\Spin_n\}\neq \{\Spin_n\}^{-1}$. This would provide the first example of a connected group $G$ for which $\{\B G\}\neq \{G\}^{-1}$, and should be at least morally related to other odd behaviour of spin groups, regarding for example essential dimension \cite{BRV} and (conjectural) failure of stable rationality of fields of invariants \cite[Conjecture 4.5]{merk}. We plan to return to this point in future work.

\subsection*{Notations and conventions}
We will work over a {field $k$ of characteristic different from $2$, and containing a square root of $-1$, that we denote by $\sqrt{-1}$.}

\subsection*{Acknowledgements}
We are happy to thank Zinovy Reichstein for several very useful conversations and comments on a first draft. We are also grateful to Daniel Bergh for some clarifications about parts of \cite{bergh-motivic-classes}, {and to Burt Totaro {and to the anonymous referee} for useful comments}.

\section{Preliminaries}

\subsection{Clifford algebras, (s)pin groups and $G_2$}\label{sec:clifford}

Let us briefly recall the definition of the Clifford algebra and the {split} (s)pin groups. For details, we refer the reader to the notes \cite{vistoli-spin} and references therein.

The Clifford algebra $C_n$ of the quadratic form $q(x_1,\hdots,x_n)=-(x_1^2+\cdots +x_n^2)$ on the vector space $V=k^n=\langle e_1,\hdots, e_n\rangle$ is the associative $k$-algebra generated by $e_1,\hdots, e_n$ with relations $e_i^2=-1$ and $e_ie_j+e_je_i=0$ for $i\neq j$. {Note that since $k$ contains a square root of $-1$, this quadratic form is equivalent to the split quadratic form $q_n$ on $k^n$ given by $q_n(x_1,\hdots, x_n)=x_1x_2+x_3x_4+\cdots +x_{n-1}x_n$ if $n$ is even and $q_n(x_1,\hdots, x_n)=x_1x_2+x_3x_4+\cdots +x_n^2$ if $n$ is odd.}
As a vector space, $C_n$ can be identified with the exterior algebra $\bigwedge^\bullet V$, but the product of $C_n$ is not the wedge product. For example, if two vectors $v,v' \in V$ are orthogonal, then they anticommute in $C_n$, i.e. $vv'+v'v=0$, but more generally, for every $v,v'\in V$ we have $vv'+v'v=2h(v,v')$, where $h$ is the symmetric bilinear form associated with $q$. In particular $v^2=q(v)$ for every $v\in V$ (and in fact this property characterizes the Clifford algebra of $q$).

On the exterior algebra $\bigwedge^\bullet V$, and hence on $C_n$, we have three $k$-linear involutions that will play a role in what follows. The first one, that we denote by $\epsilon$, is completely determined by
\begin{itemize}
\item $\epsilon(1)=1$
\item $\epsilon(v)=-v$ for $v\in V$
\item $\epsilon(a\wedge b)=\epsilon(a)\wedge \epsilon(b)$
\end{itemize}
and is also an automorphism of algebras.
The second one, denoted by $(-)^\mathrm{t}$, is completely determined by
\begin{itemize}
\item $1^\mathrm{t}=1$
\item $v^\mathrm{t}=v$ for $v\in V$
\item $(a\wedge b)^\mathrm{t}=b^\mathrm{t} \wedge a^\mathrm{t}$.
\end{itemize}
The third one is obtained by composing these two, and is denoted by $\overline{(-)}$.
These involutions respect the product of $C_n$, in the sense that for every $a,b\in C_n$ we have $\epsilon(ab)=\epsilon(a)\epsilon(b)$, $(ab)^\mathrm{t}=b^\mathrm{t}a^\mathrm{t}$ and $\overline{ab}=\overline{b}\overline{a}$.

The {split} pin group $\Pin_n$ is {a smooth affine} group scheme over $k$, whose $L$-rational points for any field extension $k\subseteq L$ consist of the subgroup of $C_n\otimes_k L$ of elements $\alpha$ such that
\begin{itemize}
\item $\alpha$ is either even or odd (for the grading of {$\bigwedge^\bullet V\otimes_k L$)},
\item $\alpha\overline{\alpha}=1$, and
\item for every {$v\in V\otimes_k L$}, the element $\alpha v \overline{\alpha}$ is also in {$V\otimes_k L$}.
\end{itemize}
The {split} spin group $\Spin_n\subset \Pin_n$ is the {sub-group scheme} of even elements.

There is a surjective homomorphism $\rho_n\colon \Pin_n \to \O_n$ with kernel $\mu_2=\{\pm 1\}$ (that restricts to $\rho_n\colon \Spin_n\to \SO_n$), defined as $\rho_n(\alpha)v=\epsilon(\alpha)v\overline{\alpha}$. If $v\in V$ is a vector of length $1$ (i.e. $\|v\|^2=1$, where $\|(x_1,\hdots, x_n)\|^2=x_1^2+\cdots+x_n^2$), then $\rho_n(v)$ is the reflection through the hyperplane orthogonal to $v$.
Every element of $\Pin_n$ can be written in $C_n$ as a product of vectors of length $1$ in $V$, and conversely any such product is an element of $\Pin_n$. The elements of $\Spin_n$ are exactly the ones of $\Pin_n$ that can be written as a product of an even number of vectors of length $1$.
The group $\Spin_n$ is a simple connected linear algebraic group, and it has the same Lie algebra as the group $\SO_n$ (of which it is the universal cover). Apart from $\Spin_1\cong \mu_2$, spin groups are special (i.e. every torsor is locally trivial for the Zariski topology) for $n\leq 6$, because of accidental isomorphisms (see for example \cite[Section 16]{Garibaldi}).

Let us also briefly recall how the algebraic group $G_2$, the smallest of the exceptional simple groups, is constructed, over the complex numbers. Recall that the octonion algebra $\mathbb{O}$ is a normed division algebra of dimension 8 as a vector space over $\bR$, constructed by applying the Cayley--Dickson construction to the algebra of quaternions $\bH$. The product is non-commutative and non-associative. The tensor product $\mathbb{O}_\bC=\mathbb{O}\otimes_\bR \bC$ is a $\bC$-algebra of dimension $8$.
The complex algebraic group $G_2$ can be defined as the group of automorphisms of $\mathbb{O}_\bC$ as a $\bC$-algebra. It has dimension $14$ and the smallest irreducible representation of rank bigger than $1$ has rank $7$, and is given by considering the natural action on the space $\Im(\mathbb{O})$ of purely imaginary octonions. 

Both the split form of $G_2$ and this $7$-dimensional representation can be defined over an arbitrary field (see for example \cite[Theorem 25.14]{involutions}).

\subsection{Representations of $\O_n$ and $\SO_n$ and stabilizers}\label{sec:reps.on}

Here we recall some facts about the orbits and stabilizers for the tautological representation of {the split algebraic groups} $\O_n$ and $\SO_n$, that will be useful in the main body.

Assume $n\geq 2$, and let us consider the tautological representation of $\O_n$ on $V=k^n$ preserving the split quadratic form $q=q_n$ defined as $q_n(x_1,\hdots, x_n)=x_1x_2+x_3x_4+\cdots+x_{n-1}x_n$ if $n$ is even and $q_n(x_1,\hdots, x_n)=x_1x_2+x_3x_4+\cdots+x_n^2$ if $n$ is odd.
Denote by $C$ the punctured null-cone $\{0\neq v\in V \mid q(v)=0\}$, by $B$ the complement {$V\setminus \overline{C}=\{v\in V \mid q(v)\neq 0\}$},  and by $Q$  the non-singular quadric $\{v\in V\mid q(v)=1\}$. Both $C$ and $B$ are unions of orbits, and moreover the action of both $\O_n$ and $\SO_n$ is transitive on $C$ and $Q$ (except for $\SO_2$), by Witt's extension theorem. In fact, orbits for the action are exactly the origin $\{0\}$, the locus $C$, the non-singular quadric $Q$, and the other orbits are all isomorphic to $Q$ via rescaling (and contained in $B$).

We will need a description of stabilizers of points on $C$ and $Q$ for these actions. The stabilizer of a point $p\in Q$ for the action of $\O_n$ is a copy of $\O_{n-1}$. Let us denote the embedding by $i\colon \O_{n-1}\subset \O_n$.

Let us analyze the stabilizer for the action of $\O_n$, say, of the element $e_1 \in C$.  Denote this stabilizer by $G$, consider the subspace $W=\langle e_3, \hdots,  e_n\rangle\subset V$, with the induced quadratic form, and let us also identify $\O(W)\cong \O_{n-2}$. Note that $\O_{n-2}$ is included in $G$ in the obvious manner, but one can also easily check that this inclusion has a section $G\to \O_{n-2}$: given an element $g \in G$ (i.e. $g(e_1)=e_1$), we can include $W\subset V$, then apply $g\colon V\to V$, and project down to $\langle e_1 \rangle^\perp/\langle e_1\rangle\cong W$. A straightforward computation shows that the resulting linear transformation $W\to W$ is in $\O_{n-2}$.

Hence we obtain a short exact sequence
$$
\xymatrix{
0\ar[r] & K\ar[r] & G\ar[r] & \O_{n-2}\ar[r] & 0
}
$$
and an easy computation (that we leave to the reader) shows that the group $K$ is isomorphic to $W$, seen as an algebraic group via its linear structure.

Precisely, the image in $G$ of an element $w\in W$ is the linear transformation $\phi_w\colon V\to V$ such that
\begin{itemize}
\item $\phi_w(e_1)=e_1$
\item $\phi_w(e_2)=-q(w)e_1+e_2+w$
\item $\phi_w(x)=x-{2}h(x,w)e_1$ for $x\in W$.
\end{itemize}
(where $h$ is the symmetric bilinear form on $V$ associated with $q$ as $h(v,v')=\frac{1}{2}(q(v+v')-q(v)-q(v'))$, so that $h(v,v)=q(v)$ for every $v\in V$).

Since $G\to \O_{n-2}$ has a section, we conclude that $G\cong \O_{n-2}\ltimes W$, where the action of $\O_{n-2}$ on $W$ is given by the identification $\O_{n-2}=\O(W)$. The argument for $\SO_n$ is completely analogous, and gives an isomorphism $G\cong \SO_{n-2}\ltimes W$.

We will also need to consider the action of $\O_n\times \mu_2$ (resp. $\SO_n\times \mu_2$) on $V$, where $\mu_2$ acts by multiplication by $-1$. The stabilizer in $\O_n\times \mu_2$ of a point of $Q$ for this action is the image of the map $i'\colon \O_{n-1}\times \mu_2 \subset \O_n\times \mu_2$ given by $i'(M,\xi)=(\xi \cdot i(M),\xi)$. The stabilizer for $\SO_n\times \mu_2$ is the inverse image in $\SO_n\times \mu_2$ of the image of $i'$, so it consists of pairs $(M,\xi)\in \O_{n-1}\times \mu_2$ such that the matrix $\xi\cdot i(M)$ is in $\SO_n$.

The determinant of $\xi\cdot i(M)$ is $(\xi)^n\cdot \det(M)$. If $n$ is even this is $\det(M)$, and we conclude that $M$ has to be in $\SO_{n-1}$ and the stabilizer is isomorphic to $\SO_{n-1}\times \mu_2$ (which, since $n-1$ is odd, is isomorphic to $\O_{n-1}$) included in $\SO_n\times \mu_2$ via the restriction $\SO_{n-1}\to \SO_n$ of the map $i$. If $n$ is odd, the determinant is $\xi\cdot \det(M)$, and then the stabilizer is given by pairs $(M,\xi)\in \O_{n-1}\times \mu_2$ such that $\xi=\det(M)$, which is isomorphic to $\O_{n-1}$ via the second projection.

\subsection{The Grothendieck ring of stacks}

We briefly recall Ekedahl's definition of the Grothendieck ring $\kstack$ of algebraic stacks over $k$, and give a list of facts about it that will be needed later. For more details we refer the reader to \cite{ekedahl-grothendieck}.

The ring $\kstack$ is generated as an abelian group by the isomorphism classes $\{X\}$ of algebraic stacks of finite type over $k$ with affine stabilizer groups, with relations $\{X\}=\{Y\}+\{X\setminus Y\}$ for $Y\subseteq X$ a closed substack, and $\{E\}=\{\bA^n\times_k X\}$ for a vector bundle $E\to X$ of rank $n$. The product is defined by fiber product over $k$, i.e. $\{X\}\cdot \{Y\}=\{X\times_k Y\}$. The relations for vector bundles are automatic in the Grothendieck ring of varieties, but they have to be imposed in this case. We will denote by $\LL$ the so-called \emph{Lefschetz motive}, the class of the affine line $\bA^1$. With this notation, if $E\to X$ is a vector bundle of rank $n$, then $\{E\}=\LL^n\{X\}$. 

By \cite[Theorem 1.2]{ekedahl-grothendieck} the ring $\kstack$ is isomorphic to the localization of the Grothendieck ring of varieties obtained by inverting $\LL$ and all elements of the form $\LL^n-1$ for $n \in \bN_+=\bN\setminus \{0\}$. In particular, products of powers of $\LL$ and cyclotomic polynomials in $\LL$ are invertible in $\kstack$.

\begin{proposition}[{\cite[Proposition 1.4]{ekedahl-grothendieck}}]
\label{prop:bg}
Let $G$ be a special algebraic group over $k$. Then for every $G$-torsor $X\to Y$ we have $\{X\}=\{G\}\cdot \{Y\}$ in $\kstack$. In particular $\{\B G\}=\{G\}^{-1}$. \qed
\end{proposition}

\begin{proposition}\label{prop:linear.rep}
Let $G$ be an algebraic group over $k$ that acts linearly on a vector space $V$ of dimension $n$ over $k$. Then in $\kstack$ we have $\{[V/G]\}=\LL^n \{\B G\}$.
\end{proposition}

\begin{proof}
This follows from the construction of $\kstack$, by noting that the natural map $[V/G]\to \B G$ is a vector bundle of rank $n$.
\end{proof}

Denote by $\Phi_\LL$ the submonoid of the polynomial ring $\bZ[\LL]$ generated by (non-negative) powers of $\LL$ and cyclotomic polynomials $\LL^n-1$ for $n\in \bN_+$.
The following proposition and its corollary are used in \cite{bergh-motivic-classes} (see in particular Section 2.2).

\begin{proposition}\label{prop:equal}
Let $G$ be an affine connected algebraic group of finite type over $k$. If both classes $\{\B G\}$ and $\{G\}$ are in the subring $\Phi_\LL^{-1}\bZ[\LL]\subset \kstack$, then we have $\{\B G\}\cdot \{G\}=1$ in $\kstack$.
\end{proposition}

\begin{proof}
Consider the ``Euler characteristic'' ring homomorphism $\chi_c\colon \kstack\to \widehat{\mathrm{K_0}}(\mathrm{Coh}_k)$ constructed in \cite[Section 2]{ekedahl-grothendieck}, where recall that $\mathrm{K_0}(\mathrm{Coh}_k)$ is a certain Grothendieck ring of mixed Galois representations (or Hodge structures if $k=\bR$ or $\bC$). Since, as explained in [loc. cit.], this factors through a map $\mathrm{K_0^\mathcal{G}}(\mathrm{Stack}_k)\to \widehat{\mathrm{K_0}}(\mathrm{Coh}_k)$ (where in $\mathrm{K_0^\mathcal{G}}(\mathrm{Stack}_k)$ we also impose that $\{X\}=\{G\}\cdot \{Y\}$ for every connected algebraic group $G$ over $k$ and every $G$-torsor $X\to Y$), we deduce that $\chi_c(\{\B G\}\cdot \{G\})=1$ in $ \widehat{\mathrm{K_0}}(\mathrm{Coh}_k)$. The conclusion now follows from \cite[Lemma 3.5]{ekedahl-finite-group}.
\end{proof}

\begin{corollary}\label{cor:rat.function}
Let $G$ be an affine connected algebraic group of finite type over $k$, such that $\{G\}$ is in the subring $\Phi_\LL^{-1}\bZ[\LL]\subset \kstack$.
Then $\{\B G\}=\{G\}^{-1}$ if and only if $\{\B G\}\in \Phi_\LL^{-1}\bZ[\LL]$.  \qed
\end{corollary}

\begin{remark}
The assumption about $\{G\}$ in the previous statement is satisfied for $G$ split semisimple, see \cite[Proposition 2.1]{behrend-dhillon}. {The same is true for more general groups: if $G$ is split reductive, then there is a split torus $T$ and a $T$-principal bundle $G\to G^{\mathrm{ss}}$ where $G^{\mathrm{ss}}$ is split semisimple, hence we can again conclude that $\{G\} \in \Phi_\LL^{-1}\bZ[\LL]$. Moreover, this is also true  if $G$ is connected with split Levi quotient, and the characteristic of $k$ is zero: in that case, if $U\subseteq G$ denotes the unipotent radical, then $U$ is special (being an iterated extension of copies of $\bG_a$) and we have a $U$-principal bundle $G\to L$ where $L$ is a split reductive group.

For non-split groups, the situation is more complicated (already for tori, see \cite{rok}).
}
\end{remark}

\begin{proposition}[{\cite[Proposition 2.2]{talpo-vistoli}}]
\label{prop:linear}
Let $G$ be an affine algebraic group over $k$ acting linearly on a $d$-dimensional vector space $V$, that we also see as a group scheme via addition. Then we have
$$
\{\B(V\rtimes G)\}=\LL^{-d}\{\B G\}
$$
in $\kstack$.\qed
\end{proposition}

\begin{proposition}[{\cite[Corollary 3.9]{ekedahl-finite-group}}]
\label{prop:unipotent}
Let $U$ be a unipotent algebraic group over $k$ of dimension $d$. Then we have
$$
\{\B U\}=\LL^{-d}
$$
in $\kstack$.\qed
\end{proposition}

\begin{proposition}\label{prop:formulas}
In $\kstack$ we have the following formulas.
$$
\{G_2\}=\LL^{14}(1-\LL^{-2})(1-\LL^{-6})=\LL^6(\LL^2-1)(\LL^6-1)
$$
$$
\{\Spin_n\}= \begin{cases}\displaystyle
     \LL^{m^2- m}(\LL^m - 1)\prod_{i=1}^{m-1} (\LL^{2i}-1)
 &\text{if }\/n = 2m
   \vspace{3pt}\\
   \displaystyle
   \LL^{m^{2}} \prod_{i = 1}^{m}(\LL^{2i} - 1)&\text{if }\/n = 2m + 1
   \end{cases}
$$
\end{proposition}

\begin{proof}
These follow immediately from \cite[Proposition 2.1]{behrend-dhillon}. For $G_2$ we have $\dim G=14$, $d_1=2$ and $d_2=6$, and for $\Spin_n$, we have $\dim G=n(n-1)/2$ and the integers $d_i$ are $\{2, 4, \hdots, 2m\}$ if $n=2m+1$ and $\{2, 4, \hdots, 2m-2, m\}$ if $n=2m$ (note that $\Spin_n$ and $\SO_n$ share the same Lie algebra).  
\end{proof}

\section{The class of $\B G$ for $G=G_2, \Spin_7, \Spin_8$}\label{sec:g2}

In this section we will compute the class of the classifying stack $\B G$ in the Grothendieck ring $\kstack$ for $G=G_2, \Spin_7, \Spin_8$, and check that in each case it is equal to $\{G\}^{-1}$.

\subsection{General setup}\label{sec:setup}

Let us explain the basic strategy that we will employ several times, in this section and the next one, to obtain information on the class $\{\B G\}$ for a group $G$ acting on an $n$-dimensional vector space $V$ via a homomorphism $G\to \O_n$ (resp. $G\to \SO_n$).

Let us stratify the space $V$ by considering the origin $\{0\}$, the locus $C=\{0\neq v\in V\mid q(v)=0\}$ and the complement $B=V\setminus \overline{C}$, as in (\ref{sec:reps.on}). By Proposition \ref{prop:linear.rep}, we have that $\{[V/G]\}=\LL^n\{\B G\}$. On the other hand, since the loci $\{0\}, C, B$ are clearly $G$-invariant, we obtain
$$
\{[V/G]\}=\{[\{0\}/G]\}+\{[C/G]\}+\{[B/G]\}.
$$
Moreover, noting that $\{[\{0\}/G]\}=\{\B G\}$, we can rewrite the resulting equality as
$$
\{\B G\}(\LL^n-1)=\{[C/G]\}+\{[B/G]\}.
$$
Now assume that the action of $G$ on $C$ and on $Q$ is transitive. This will be the case for example if $G\to \O_n$ (resp. $G\to \SO_n$) is surjective.

Call $G'$ the stabilizer of a point of $C$. Then we have 
$$
\{\B G\} (\LL^n-1)=\{\B G'\}+\{[B/G]\}.
$$
In order to deal with the second term, we employ the following construction (that was first used in \cite{molina-vistoli-classifying}, and subsequently for example in \cite{tesi-molina}, \cite{guillot} and \cite{talpo-vistoli}). Let $Q$ be the closed locus of points $v\in V$ such that $q(v)=1$. We can consider the product $\gm\times Q$ and the natural multiplication map $\gm\times Q \to B$. There is an action of $\mu_2$ on the total space given by $\xi\cdot (\lambda, v)=(\xi\lambda, \xi v)$, and for this action we have $(\gm\times Q)/\mu_2\cong B$. Moreover $G$ acts on $\gm\times Q $ (trivially on $\gm$) and this action commutes with the action of $\mu_2$. We obtain an action of $G\times \mu_2$ on $\gm\times Q $, and an isomorphism
$$
[B/G]\cong [(\gm\times Q )/(G\times \mu_2)].
$$
Thus we get 
$$
\{[B/G]\}=\{ [(\gm\times Q)/ (G\times \mu_2) ]\} = (\LL-1) \{ [Q/(G\times \mu_2)]\}
$$
since $[(\gm\times Q)/ (G\times \mu_2) ]\to [Q/(G\times \mu_2)]$ is a $\gm$-torsor.

The action of $G\times \mu_2$ on $Q$ is transitive. Call $G''\subset G\times \mu_2$ the stabilizer of a point of $Q$ for this action, so that $ [Q/(G\times \mu_2)]\cong \B G''$.
Overall we obtain
$$
\{\B G\}(\LL^n-1)=\{\B G'\}+(\LL-1)\{ \B G''\}
$$
and, since $\LL^n-1$ is invertible in $\kstack$, this formula allows us to compute $\{\B G\}$ whenever we understand $\{\B G'\}$ and $\{\B G''\}$.

If the action of $G$ is transitive only on $C$ or only on $Q$, we will use the corresponding part of the preceding arguments.

\subsection{The class of $\B G_2$}

\begin{theorem}\label{prop:g2}
We have $\{\B G_2\}=\{G_2\}^{-1}$ in $\kstack$.
\end{theorem}

We will prove this in several steps, by applying the strategy outlined in (\ref{sec:setup}) a few times.

We start by considering the representation of $G_2$ on $k^7$ via $\O(7)$ (for the standard quadratic form $q(x_1,\hdots,x_7)=\sum_{i=1}^{7} x_i^2$). Over the complex numbers, this is given by identifying $\bC^7$ with the space $\Im(\mathbb{O})$ of purely imaginary octonions.

\begin{lemma}\label{lemma:group.g2}
We have
\begin{equation}
\label{eq:1}
\lbrace \B G_2 \rbrace = (\LL^7-1)^{-1}\Big (\LL^{-6}(\LL^2-1)^{-1}+(\LL-1)\{ \B G \}\Big )
\end{equation}
in $\kstack$, where $G\subset G_2\times \mu_2$ is the stabilizer of $e_1 \in k^7$ (where the action of $\mu_2$ is by scaling). 
\end{lemma}

\begin{proof}

As explained in (\ref{sec:setup}), we will stratify $k^7$ in invariant subvarieties, by considering as strata the origin $\{0\}$, the punctured null-cone $C_7=\{ 0\neq v \in k^7 \mid q(v)=0\}$, and the complement $B_7=k^7\setminus \overline{C_7}$.

From this stratification, using the fact that $[\{0\}/G_2]\cong \B G_2$ and $\{[k^7/G_2]\}=\LL^7 \{\B G_2\}$ 
we obtain
\begin{equation}
\label{eq:2}
\{\B G_2\}(\LL^7-1)=\{[C_7/G_2]\}+\{[B_7/G_2]\}.
\end{equation}
We will handle the two terms on the right hand side separately.

{For the first term, the action of $G_2$ on $C_7$ is transitive: see \cite[Proposition 6.3]{guillot} for the case of characteristic $0$, and \cite[Section 4.3.5]{wilson} for odd characteristic. Both references deal with specific fields (the field of complex numbers $\bC$ and any finite field $\bF_q$ with $q$ odd, respectively), but standard arguments and the fact that the action is defined on the prime field give the same conclusion for an arbitrary field $k$ of characteristic not $2$.} Call $G'\subset G_2$ the stabilizer of a point of $C_7$, so that $[C_7/G_2]\cong \B G'$. If we projectivize the situation, we find that $\bP(C_7)\cong G_2/P$ for a parabolic subgroup $P$ of $G_2$ that fits into a short exact sequence
$$
\xymatrix{
0\ar[r] & G'\ar[r] & P\ar[r] & \gm\ar[r] & 0.
}
$$
Moreover, the Levi decomposition of $P$ is $H\rtimes \GL_2$, where $H$ is unipotent of dimension $5$ (see \cite[Section 3]{zinovyg2}).

Now since $\B G'\to \B P$ is a $\gm$-torsor, we find $\{\B G'\}=(\LL-1) \{\B P\}$, and since $\B H\to \B P$ is a $\GL_2$-torsor, we find $\{\B H\}=\{\GL_2\}\cdot \{\B P\}$. Since $H$ is unipotent of dimension $5$, by Proposition \ref{prop:unipotent} we obtain $\{\B H\}=\LL^{-5}$, and it is well known (see for example \cite[Proposition 1.1]{ekedahl-grothendieck}) that
$$
\{\GL_n\}=(\LL^n-1)(\LL^n-\LL)\cdots (\LL^n-\LL^{n-1}).
$$
By using these formulas we obtain
\begin{equation} 
\label{eq:3}
\{[C_7/G_2]\}=\{\B G'\}=(\LL-1)\LL^{-5}\{\GL_2\}^{-1}=\LL^{-6}(\LL^2-1)^{-1}
\end{equation}
which is the first term inside the parentheses in (\ref{eq:1}).

As for the second term, as explained in (\ref{sec:setup}) we have
$$
\{[B_7/G_2]\}=(\LL-1)\{[Q_7/(G_2\times \mu_2)]\}
$$
where $Q_7$ is the locus of points $v\in k^7$ such that $q(v)=1$. {Now observe that the action of $G_2$ on $Q_7$ is transitive (see again \cite[Proposition 6.2]{guillot} for characteristic $0$ and \cite[Section 4.3.3]{wilson} for odd characteristic, or \cite[Proposition 1.4 (iii)]{macdonald} and \cite[\S 36 Exercise 6]{involutions})},
so we have $[Q_7/(G_2\times \mu_2)]\cong \B G$, where $G$ is the stabilizer of $e_1$, as in the statement of the Lemma.  

Combining
$$
\{[B_7/G_2]\}=(\LL-1)\{[Q_7/(G_2\times \mu_2)]\}=(\LL-1)\{\B G\}
$$
with equations (\ref{eq:2}) and (\ref{eq:3}) concludes the proof.
\end{proof}

We now turn our attention to the group $G$. We will consider the natural action of $G$ on the orthogonal complement of $e_1$ in $k^7$, via the group $\O(6)$, and use similar arguments. The following lemma gives a convenient description of the group $G$, that will allow us
to get a better grip on this action

\begin{lemma}\label{lemma:description.group.g}
The group $G$ is isomorphic to the semidirect product $SL_3 \rtimes \mu_2$, where the morphism $\mu_2 \rightarrow \aut (\SL_3)$ is given by $(-1)\cdot (A)=(A^{\rT})^{-1}$.
\end{lemma}

\begin{proof}
By \cite[Proposition 1.4 (iii)]{macdonald} the stabilizer of $e_1$ for the action of $G_2$ on $k^7$ is a copy of $\SL_3$, and this gives an embedding $\SL_3\subset G$, which is clearly the kernel of the (surjective) natural projection $G\subset G_2\times \mu_2 \to \mu_2$. This gives a short exact sequence 
$$
\xymatrix{
0\ar[r] & \SL_3 \ar[r] & G \ar[r] & \mu_2 \ar[r] & 0
}
$$
that we can split by considering the element of $G_2$ of order $2$ that sends $e_1$ to $-e_1$, given by $\sigma \in G_2$ that acts on the standard basis $\{e_1,\hdots, e_7\}$ (which can be thought of as the imaginary units of the octonions, if $k=\bC$) by $\sigma(e_i)=(-1)^ie_i$. Here we are using the following notation for the octonions: $e_0=1$ is the unit, $e_1,e_2, e_3$ span a quaternion subalgebra, and can be identified with the units of the quaternions $i,j,k$ respectively, $e_4$ is another independent element with square $-1$, and the rest are the products $e_5=e_1e_4$, $e_6=e_2e_4$,  and $e_7=e_3e_4$.

The assertion about the resulting homomorphism $\mu_2\to \aut(\SL_3)$ follows from the description of the embeddings $\SL_3\subset G_2$ given in \cite[Section 2.1]{chloe-pauly}.
\end{proof}

Let $V=k^3$ be the standard $3$-dimensional representation of $\SL_3$. We can define a $6$-dimensional representation of $G=\SL_3\rtimes \mu_2$ on $V\oplus V$ via the formulas
$$A\cdot (v,w)=(Av, (A^{\rT})^{-1}w), \quad (\id, -1) \cdot (v, w)=(w,v),$$ so that for a general element $(A,\xi) \in G$ with $\xi\in \mu_2$, we have $$(A,\xi)\cdot (v,w)=A\cdot (\xi(v,w))=\begin{cases} A\cdot (v,w)\quad  \mbox{ if } \xi=1 \\   A\cdot (w,v) \quad \mbox{ if } \xi=-1 \end{cases}$$

\begin{lemma}\label{lemma:group.g}
We have
$$
\{\B G\}=(\LL^6-1)^{-1}\Big( \LL^{-3}+\LL^{-3}(\LL^2-1)^{-1} + (\LL-1)\{\B H\}\Big)
$$
in $\kstack$, where $H\subset G\times \mu_2$ is the stabilizer of $(e_1,e_1)\in V\oplus V$ (and $\mu_2$ acts by scaling, as usual).
\end{lemma}

\begin{proof}

On $V\oplus V$, we consider the quadratic form $q(v,w)=v^{\rT}w$. The action of $G$ fixes the associated symmetric bilinear product, so it factors as $G\to \O(6)\to \GL_6$. We stratify the vector space $V\oplus V$ using the quadratic form. We define the loci

\begin{itemize}
\item $B_6=\lbrace (v,w) \in V\oplus V \mid q(v,w) \neq 0 \rbrace$
\item $C_6=\lbrace (v,w) \in V\oplus V \mid v,w\neq 0 \mbox{ and } q(v,w)=0 \rbrace$
\item $D_6=V\times 0 \cup 0\times V \smallsetminus \lbrace (0,0) \rbrace$
\end{itemize}
so that we can write
\begin{equation}
\label{eq:5}
\{\B G\}(\LL^6-1)=\{[C_6/G]\}+\{[D_6/G]\}+\{[B_6/G]\}.
\end{equation}
Contained in $B_6$ we also have the smooth quadric $Q_6=\lbrace (v,w)\in V\oplus V \mid q(v,w) =1 \rbrace$. We can apply the usual trick explained in (\ref{sec:setup}) to obtain that 
$$
\{ [B_6/G ]\}=(\mathbb{L}-1)\lbrace [ Q_6/G \times \mu_2 ] \rbrace
$$
Now note that the action of $G\times \mu_2$ on $Q_6$ is transitive: let $(v,w)$ be a point of $Q_6$, and note that we can pick $A \in \SL_3$ such that $Av=e_1$, so we can assume $v=(e_1,u)$, where $u=(1,a,b)$. It suffices to show that $(e_1,e_1)$ can be mapped to $(e_1,u)$ for all $u=(1,a,b)$. To do so we can just pick the matrix
\[
\left( \begin{array}{ccc} 1 & -a & -b\\
0 & 1 & 0\\
0 & 0 & 1 

\end{array} \right).
\] 
This shows that the action is transitive. 

Let $H\subset G\times \mu_2$ be the stabilizer of $(e_1,e_1)$. Then we deduce
\begin{equation}
\label{eq:6}
\{ [B_6/G ]\}=(\LL-1)\lbrace [ Q_6/G \times \mu_2 ] \rbrace = (\LL-1)\{\B H\}.
\end{equation}
Let us turn to the loci $C_6$ and $D_6$. The situation on $D_6$ is simple: the action is clearly transitive, and the stabilizer $G'$ of an element (pick $(e_1,0)$, for example) is isomorphic to the group of matrices of the form
\[
\left( \begin{array}{ccc} 1 & a & b\\
0 & c & d\\
0 & e & f 
\end{array} \right)
\]
where the $2\times 2$ lower right block is in $\SL_2$. The group $G'$ has a normal subgroup isomorphic to $\ga^2$, given by the matrices as above for which $c=f=1$ and $d=e=0$. The quotient is isomorphic to $\SL_2$. 

From this, since $\B\ga^2 \to \B G'$ is an $\SL_2$-torsor, and both $\ga^2$ and $\SL_2$ are special, we have
\begin{equation}
\label{eq:7}
\{[D_6/G]\}=\{\B G'\}=\{\SL_2\}^{-1}\LL^{-2}=\LL^{-3}(\LL^2-1)^{-1}
\end{equation}
where we used the formula $\{\SL_2\}=\LL(\LL^2-1)$, that follows from $\{\GL_2\}=(\LL^2-1)(\LL^2-\LL)$ and the fact that the determinant $\GL_2\to \bG_m$ is an $\SL_2$-torsor.

Finally let us consider the action on $C_6$. We first claim that this action is transitive. Pick an element $(v,w)$ in $C_6$. As before we can assume that $v=e_1$, and it suffices to show that it is conjugate to $(e_1,e_2)$. Suppose first that $\| w \| \neq 0$ (where we denote by $\| w \|$ a chosen solution $x\in k$ of the equation $x^2=w_1^2+w_2^2+w_3^2$). Set $w'=\|w\|^{-1}w$, and let $e_1, w', u$ be a positively oriented orthonormal basis. Then there is an element $A \in \SO_3 \subset \SL_3$ sending $e_1, e_2, e_3$ to $e_1, w', u$. As $(A^{\rT})^{-1}=A$, we get $A(e_1,e_2)=(e_1,w')$. We can now pick an \emph{ad hoc} diagonalizable matrix with determinant $1$ to send $(e_1, w')$ to $(e_1,w)$.

Now let $w$ be an isotropic vector. As $w$ is orthogonal to $e_1$, it must be equal to $a e_2 + b e_3$ for some $a,b \neq 0$ with $a^2+b^2=0$. Up to multiplying by an invertible diagonal matrix with determinant $1$ (that might move $e_1$ and send it to $\alpha  e_1$) we can assume $a=1, b=\sqrt{-1}$. Consider the matrix

\[ A=
\left( \begin{array}{ccc} 1 & 0 & 0\\
0 & 1 & 0 \\

0 & -\sqrt{-1} & 1
\end{array} \right).
\] 

The matrix $A$ sends $w=e_2+\sqrt{-1}e_3$ to $e_2$ and has determinant $1$. As $(A^{\rT})^{-1}$ sends $e_1$ to itself, we have $(A^{\rT})^{-1}(\alpha e_1,w)=(\alpha e_1,e_2)$. By acting with a diagonal matrix with entries $\alpha^{-1}, 1 , \alpha$, we can get to $(e_1,e_2)$. This shows that the action on $C_6$ is transitive.

Consider now the stabilizer of the point $(e_1,e_3)$. It is not hard to check that any $A \in \SL_3$ that is an upper triangular matrix whose diagonal entries are $1$ fixes $(e_1,e_3)$, and that these are the only elements of $\SL_3$ that do so. This does not describe the whole stabilizer in $\SL_3\rtimes \mu_2$, because we also have elements of the form $(B,-1)$ where $B(e_1,e_3)=(e_3,e_1)$ (for example, if $B$ sends $e_1,e_2, e_3$ to $e_3, -e_2, e_1$ respectively).

More precisely, if $U$ is the unipotent group of upper triangular matrices with entries equal to $1$ on the diagonal, then the stabilizer of $(e_1,e_3)$ in $\SL_3\rtimes \mu_2$ is isomorphic to the semidirect product $U\rtimes \mu_2$ (where the projection to $\mu_2$ corresponds to the projection $\SL_3\rtimes \mu_2\to \mu_2$, and the group $U$ is the kernel of this map). The action of $\mu_2$ on $U$ is determined by sending $-1\in \mu_2$ to the automorphism
\[
\left( \begin{array}{ccc} 1 & a & b\\
0 & 1 & c\\
0 & 0 & 1 
\end{array} \right) \mapsto \left( \begin{array}{ccc} 1 & c & ac-b\\
0 & 1 & a\\
0 & 0 & 1 
\end{array} \right)
\]
of $U$. By Proposition \ref{prop:unipotent}, the class of $\B U$ is $\LL^{-3}$.

In order to compute the class $\{\B(U\rtimes \mu_2)\}$, we will consider the linear action of $U\rtimes \mu_2$ on $V\oplus V$ itself (induced by the inclusion $U\rtimes \mu_2\subset \SL_3\rtimes \mu_2$), and stratify the space in invariant subspaces whose class we can compute.

\begin{lemma}\label{lemma:bu}
The class of $\B(U\rtimes\mu_2)$ is $\LL^{-3}$.
\end{lemma}

After we have proved this lemma, putting $\{[C_6/G]\}=\{\B(U\rtimes \mu_2)\}=\LL^{-3}$ together with (\ref{eq:5}), (\ref{eq:6}) and (\ref{eq:7}) concludes the proof of Lemma \ref{lemma:group.g}.
\end{proof}

\begin{proof}[Proof of Lemma \ref{lemma:bu}]
We start by considering the action of $U$ on $V \oplus V$. 
Consider the following locally closed subsets of $V$:
\begin{itemize}
\item $A_1=\lbrace (a,b,\lambda)  \rbrace$,  $A_2=\lbrace (a,\lambda,0) \rbrace$, $A_3=\lbrace (\lambda,0,0)  \rbrace$, and
\item $B_1=\lbrace (\gamma,a',b')  \rbrace$, $B_2=\lbrace (0,\gamma,a') \rbrace$, $B_3=\lbrace (0,0,\gamma) \rbrace$
\end{itemize}
where $\lambda, \gamma \in k^\times$ and $a,b,a',b' \in k$.

Each of the sets $A_i$ is invariant under the action of $U$ on the first copy of $V$ in $V\oplus V$, and each of the $B_j$ is invariant under the inverse transpose action on the second copy of $V$. Thus the products $A_i \times B_j \subset V \oplus V$ are invariant under the action of $U$. We have $\cup_{i=1}^{3} A_i=V \smallsetminus \lbrace 0 \rbrace$ and $\cup_{j=1}^{3} B_j=V \smallsetminus \lbrace 0 \rbrace$.

Now consider the action of $U\rtimes \mu_2$ on $V\oplus V$. We are going to stratify $V \oplus V$ as a union of products of the $A_i$ and $B_j$. Note that the action of $\mu_2$, whose non-trivial element acts by  $$((a,b,c),(c,d,e))\mapsto ((e,-d,c),(c,-b,a)),$$ sends $A_i$ to $B_i$ and vice-versa, so in our stratification we will have some strata that are the union of two components $A_i \times B_j \cup A_j \times B_i$ with $i\neq j$, and some strata that are a single component $A_i \times B_i$. We will compute $\{B(U\rtimes \mu_2)\}$ by computing the class of the quotient stacks of these strata by the action of $U\rtimes \mu_2$.

The first type of strata are the easiest to treat, as we only have to understand the action of $U$ on one of the two components, so we will start from these. We have a $U$-equivariant isomorphism $A_i \cong B_i$, so for each of these strata the two projections to $\GG_m$, given by $\lambda$ and $\gamma$ on one component and the other one in the second component, are invariant maps, giving us a projection from the quotient to $\GG_m^2$, which always has a section.

\begin{itemize}
\item[($S_1$)] $A_1 \times B_3 \cup A_3 \times B_1$: each orbit is determined by its image in $\GG_m^2$, with stabilizer $\GG_a$, so $$\left[ (A_1 \times B_3 \cup A_3 \times B_1) /( U \rtimes \mu_2) \right] \cong \GG_m^2\times \B\GG_a.$$
\item[($S_2$)] $A_1 \times B_2 \cup B_1 \times A_2$: each orbit has an element of the form $( (0,0,\lambda),(\gamma,0,a') )$ with stabilizer $\GG_a$, so $$\left[(A_1 \times B_2 \cup A_2 \times B_1)/( U \rtimes \mu_2) \right] \cong \GG_m^2\times \GG_a \times \B\GG_a.$$
\item[($S_3$)] $A_2 \times B_3 \cup A_3 \times B_2$: each orbit is determined by its image in $\GG_m^2$, with stabilizer $\GG_a^2$, so $$\left[ (A_2 \times B_3 \cup A_3 \times B_2 ) / (U \rtimes \mu_2) \right] \cong \GG_m^2\times \B\GG_a^2.$$
\item[($S_4$)] $(V\smallsetminus \{0\})\times 0 \cup 0 \times (V\smallsetminus \{0\})$: the quotient is isomorphic to $\left[(V \smallsetminus \{0\}) / U\right]$.
\end{itemize}
The remaining strata are a bit more bothersome and need to be stratified themselves. The problem here is that some elements, such as $((1,0,0),(0,0,1))$, will be fixed by the action of $\mu_2$. Denote by $\Delta_i$ the subset of elements with $\lambda=\gamma$ in $A_i\times B_i$.

Outside of the $\Delta_i$ we have an invariant map to $(\GG_m^2\smallsetminus \Delta)/\mu_2$ given by $\lambda$ and $\gamma$, where the action of $\mu_2$ either switches components or switches components and multiplies by $-1$. In either case the quotient is isomorphic to $\GG_m^2$. This projection factors through the coarse quotient space.

\begin{itemize}
\item[($S_5$)] $A_3 \times B_3 \smallsetminus \Delta_3$: each orbit is determined by its image, with stabilizer $U$.
\item[($S_6$)] $A_2 \times B_2 \smallsetminus \Delta_2$: each orbit is determined by its image, with stabilizer $\GG_a$.
\item[($S_7$)] $A_1 \times A_1 \smallsetminus \Delta_1$: the coarse quotient is $\GG_a\times \GG_m^2$, with trivial stabilizers.
\end{itemize}
Each of these maps gives rise to a gerbe which can be shown to be trivial. We will prove this for $A_2 \times B_2 \smallsetminus \Delta_2$, and a similar reasoning works for all other strata. This will show that the isomorphism classes of the quotients are respectively $\GG_m^2 \times \B U, \GG_m^2 \times \B\GG_a$ and  $\GG_a\times \GG_m^2$.

It is easy to see that using only the action of $U$, any point of $A_2 \times B_2 \smallsetminus \Delta_2$ is in the same orbit as a point of the form $((0,\lambda,0),(0,\gamma,0))$, where $(\lambda, \gamma) \in \GG_m^2\smallsetminus \Delta$. The stabilizer of such an element in $U \rtimes \mu_2$ is equal to the subgroup of matrices in $U$ of the form
$$
\left( \begin{array}{ccc} 1 & 0 & a\\
0 & 1 & 0 \\
0 & 0 & 1
\end{array} \right)
$$
which is isomorphic to $\GG_a$. Moreover, the quotient of $\GG_m^2\setminus \Delta$ by the induced action of $\mu_2=(U\rtimes \mu_2)/U$ is isomorphic to $\GG_m^2$, and since this action is free, this is an orbit space for the action of $U\rtimes \mu_2$ on $A_2 \times B_2 \smallsetminus \Delta_2$ (i.e. the fibers of the resulting map $A_2 \times B_2 \smallsetminus \Delta_2\to \GG_m^2$ are exactly the orbits). We also obtain a morphism $\left[(A_2 \times B_2 \smallsetminus \Delta_2)/(U\rtimes \mu_2) \right]\to \GG_m^2$.

Now consider the $U \rtimes \mu_2$-scheme given by $$X=(\GG_m^2\smallsetminus \Delta) \times ( U\rtimes \mu_2)$$ where $U\rtimes \mu_2$ acts trivially on the left component, and on the right component by multiplication on the left. The quotient $X/\mu_2$, where $\mu_2$ acts on the left on the first component and on the right on the second component, is an $U\rtimes \mu_2$-torsor, with quotient $(\GG_m^2\smallsetminus \Delta)/\mu_2=\GG_m^2$. The map $$X \rightarrow A_2 \times B_2 \smallsetminus \Delta_2, \quad ((\lambda, \gamma),(u,\xi))\rightarrow (u,\xi)\cdot ((0,\lambda,0),(0,\gamma,0))$$ is $\mu_2$-invariant and $U\rtimes \mu_2$-equivariant, so it descends to a $U\rtimes \mu_2$-equivariant map $X/\mu_2\to  A_2 \times B_2 \smallsetminus \Delta_2$. Overall this gives us a morphism $$ \GG_m^2 \rightarrow \left[(  A_2 \times B_2 \smallsetminus \Delta_2 )/ (U \rtimes \mu_2) \right]$$
which is a section of the projection $\left[(  A_2 \times B_2 \smallsetminus \Delta_2 )/ (U \rtimes \mu_2) \right]\to \GG_m^2$, showing that $\left[  (A_2 \times B_2 \smallsetminus \Delta_2) /( U \rtimes \mu_2) \right]$ is a trivial gerbe over $\GG_m^2$. We can now conclude by noting that the sheaf of automorphisms of the section is exactly $\GG_a$.

Now consider the quotients $\left[\Delta_i/(U \rtimes \mu_2) \right]$. They each have a map to $\GG_m$ given by either $\lambda$ or $\gamma$.

\begin{itemize}
\item[($S_8$)] $\Delta_3$: each orbit is determined by its image, with stabilizer $U\rtimes \mu_2$.
\item[($S_9$)] $\Delta_2$: each orbit is determined by its image, with stabilizer $\GG_a\rtimes \mu_2$.
\item[($S_{10}$)] $\Delta_1$: the coarse quotient is $\GG_a\times \GG_m$, the stabilizer of a point is $\mu_2$.
\end{itemize}
The same argument as above shows that these projections induce trivial gerbes. Finally note that, as $\mu_2$ acts linearly on $\GG_a$ in the second term, we know that the class of $\B(\GG_a\rtimes \mu_2)$ in $\kstack$ is going to be equal to $\LL^{-1}$ (by Proposition \ref{prop:linear}).

We are ready to conclude our computation. We want to show that

$$(\LL^6-1)\{\B(U\rtimes \mu_2)\}=\sum_{i=1}^{10}S_i=(\LL^6-1)\LL^{-3},$$
where $S_i$ denotes the class of the quotient stack of the stratum marked $(S_i)$ in the list above. Moving $S_8=(\LL-1)\{B(U\rtimes\mu_2)\}$ to the left-hand side we get
\begin{align*}
(\LL^6-\LL)\{\B(U\rtimes \mu_2 )\}& = S_1+\ldots+S_7+S_9+S_{10} \\
 & = (\LL-1)^2\LL^{-1}+(\LL-1)^2+(\LL-1)^2\LL^{-2}+(\LL^3-1)\LL^{-3} \\ & \quad +(\LL-1)^2\LL^{-3} 
 +(\LL-1)^2\LL^{-1}+(\LL-1)^2\LL  +(\LL-1)\LL^{-1}\\ & \quad +\LL(\LL-1)\\ & =  (\LL^6-\LL)\LL^{-3}
\end{align*}
as we claimed.
\end{proof}

Finally, let us focus on the group $H\subset G\times \mu_2 \subset (G_2\times \mu_2)\times \mu_2$ of Lemma \ref{lemma:group.g}.

\begin{lemma}
The group $H$ is isomorphic to the semidirect product $\SL_2 \rtimes (\mu_2 \times \mu_2)$ where the homomorphism $\mu_2\times \mu_2\to \aut (\SL_2)$ sends $(-1,1)$ to $A\mapsto (A^\rT)^{-1}$ and $(1,-1)$ to the automorphism acting as
\[
\left( \begin{array}{ccc}
a & b \\
c & d 
\end{array} \right) \mapsto \left( \begin{array}{ccc}
 a & -b\\
 -c & d 
\end{array} \right).
\]
\end{lemma}

\begin{proof}
Recall that $H\subset G\times \mu_2$ is the stabilizer of $(e_1,e_1)\in V\oplus V$. There is a natural projection $H\to \mu_2\times\mu_2$, which is the composite $H\subset G_2\times (\mu_2\times\mu_2)\to \mu_2\times \mu_2$, and also coincides with $H\subset G\times \mu_2=(\SL_3\rtimes \mu_2)\times \mu_2\to \mu_2\times\mu_2$, where in the last step we used the projection $\SL_3\rtimes \mu_2\to \mu_2$. The kernel of this map consists of matrices $A \in \SL_3$ that fix $e_1$, and such that $(A^\rT)^{-1}$ also fixes $e_1$. It is easy to see that this is isomorphic to $\SL_2$, embedded in $\SL_3$ as the ``lower right block''.

Thus we have a short exact sequence
$$
\xymatrix{
0\ar[r] &\SL_2\ar[r] & H\ar[r] & \mu_2\times \mu_2\ar[r] & 0.
}
$$
This is split, and the corresponding homomorphism $\mu_2\times \mu_2\to \aut(\SL_2)$ is as in the statement:
consider the subgroup of $H\subset \SL_3\rtimes \mu_2$ generated by $(\id,-1)$ 
and by $\tau=(M,1)$ where $M$ is the diagonal matrix sending $e_1, e_2, e_3$ to $-e_1, -e_2, e_3$ respectively. It is immediate to check that this subgroup is a copy of $\mu_2\times \mu_2$, and the map $\mu_2\times \mu_2\to H$ sending $(-1,1)$ to $(\id,-1)$ and $(1,-1)$ to $\tau$ is a splitting of the projection $H\to \mu_2\times \mu_2$. 
An easy computation shows that the actions of $(-1,1)$ and $(1,-1)$ on $\SL_2$ are as in the statement.
\end{proof}

\begin{lemma}\label{lemma:group.h}
We have
$$
\{\B H\}=(\LL^4-1)^{-1}\Big( 1+ \LL^{-1} +(\LL-1) \Big)=\LL^{-1}(\LL^2-1)^{-1}
$$
in $\kstack$.
\end{lemma}

\begin{proof}
In the notation of the proof of Lemma \ref{lemma:group.g}, the action of $H$ fixes the subspace $\langle (e_1,0),(0,e_1) \rangle$, so we can consider the induced action of $H$ on the orthogonal space $W\oplus W$, where $W=\langle e_2,e_3\rangle=k^2$. The restriction of the quadratic form gives a non-degenerate quadratic form on $W\oplus W$, that we still denote by $q$, and that is given again by $q(v,w)=v^\rT w$. The group $H$ acts on $W\oplus W\cong k^4$ via $\O(4)$, and we will stratify this space as usual: define
\begin{itemize}
\item $B_4=\lbrace (v,w) \in W\oplus W \mid q(v,w) \neq 0 \rbrace$
\item $C_4=\lbrace (v,w) \in W\oplus W \mid v,w\neq 0 \mbox{ and } q(v,w)=0 \rbrace$
\item $D_4=W\times 0 \cup 0\times W \smallsetminus \lbrace (0,0) \rbrace$
\end{itemize}
so that, as in the previous cases, we obtain
\begin{equation}
\label{eq:8}
\{\B H\}(\LL^4-1)=\{[C_4/H]\}+\{[D_4/H]\}+\{[B_4/H]\}.
\end{equation}
If we denote by $Q_4\subset W\oplus W$ the non-singular quadric where $q(v,w)=1$, we also have
$$\{[B_4/H]\}=(\LL-1)\{[ Q_4/(H\times \mu_2)]\}$$ where $\mu_2$ acts by scaling. 

Let us check that the action of $H$ on $Q_4$ is transitive. Let $(v,w)$ be an element of $Q_4$. As usual we can assume that $w=e_2$, so that $w=(1,a)$. It is immediate to verify that the only element $A \in \SL_2$ fixing $(e_2,w)$ is the identity. This tells us that the orbit of any $(v,w)$ in $Q_4$ must have dimension three, which means that it is open in $Q_4$. This proves transitivity. 

It is also immediate to see that the stabilizer $H'$ of an element of $Q_4$ in $H\times \mu_2=(\SL_2\rtimes (\mu_2\times \mu_2))\times \mu_2$ is a finite group of order $8$. It is not hard to check that $\{BH'\}=1$ no matter what specific group of order $8$ this turns out to be (we can apply \cite[Proposition 2.3]{ivan} for the dihedral group $D_4$ and for the quaternion group). From this we get
$$
\{[B_4/H]\}=(\LL-1)\{[ Q_4/(H\times \mu_2)]\}=(\LL-1)\{\B H'\}=\LL-1.
$$
Let us now look at $D_4$: it is clear that the action of $H$ is transitive, and that the stabilizer of any point is a semidirect product $\ga\rtimes\ \mu_2$. Therefore we obtain
$$
\{[D_4/H]\}=\{\B(\ga\rtimes \mu_2)\}=\LL^{-1}.
$$
Finally let us consider $C_4$. The action on this piece is not transitive: if $(v,w)\in C_4$ and we write $v=(a,b)$, then there exists $\lambda\in \gm$ such that $w=\lambda\cdot (-b,a)$. The map $(W\setminus \{0\}) \times \gm\to C_4$ sending $((a,b),\lambda)$ to $((a,b), \lambda(-b,a))$ is an isomorphism.
The action of $\SL_2\subset H=\SL_2\rtimes (\mu_2\times \mu_2)$ on $C_4 \cong  (W\setminus \{0\}) \times \gm$ leaves $\lambda$ fixed, whereas the action of $(\id,(-1,1))$ sends $\lambda$ to $-1/\lambda$, and the action of $(\id,(1,-1))$ sends $\lambda$ to $-\lambda$.

Let us further split $C_4$ in its two subsets $C_4^1$ where $\lambda^2=\pm 1$, and $C_4^2$ where $\lambda^2\neq \pm 1$, so that
$$
\{[C_4/H]\}=\{[C_4^1/H]\}+\{[C_4^2/H]\}
$$
The set $C_4^1$ is a union of two orbits for the action, and the stabilizer in each case is a semidirect product $\ga\rtimes \mu_2$. Thus, $\{[C_4^1/H]\}=2\cdot \{\B(\ga\rtimes\mu_2)\}=2\LL^{-1}$.

On $C_4^2$, with arguments similar to the one used in the proof of Lemma \ref{lemma:bu} for the stratum $(S_6)$, one can show that there is an isomorphism
$$
[C_4^2/H]\cong (\bA^1\setminus \{\pm 2\})\times \B\GG_a
$$
so that
$
\{[C_4^2/H]\}=(\LL-2)\LL^{-1}.
$
Here  $\bA^1\setminus \{\pm 2\}$ is the quotient of $\GG_m\setminus \{\pm 1, \pm \sqrt{-1}\}$ by the action of $\mu_2\times \mu_2$ described above.

Overall, this shows that
$$
\{[C_4/H]\}=2\LL^{-1}+ (\LL-2)\LL^{-1}=1
$$
and concludes the proof.
\end{proof}

We can now conclude the computation of $\{\B G_2\}$.

\begin{proof}[Proof of Theorem \ref{prop:g2}]
Plugging the result of Lemma \ref{lemma:group.h} into the formula of Lemma \ref{lemma:group.g}, we obtain
\begin{align*}
\{\B G\} & = (\LL^6-1)^{-1}\Big( \LL^{-3}+\LL^{-3}(\LL^2-1)^{-1} + (\LL-1)\LL^{-1}(\LL^2-1)^{-1}\Big)  \\
& =(\LL^6-1)^{-1}\LL^{-3}(\LL^2-1)^{-1}\Big(   \LL^2-1+1+(\LL-1)\LL^{2}    \Big)\\
& = (\LL^6-1)^{-1}(\LL^2-1)^{-1}.
\end{align*}
Finally, plugging this value into the formula of Lemma \ref{lemma:group.g2} we obtain
\begin{align*}
\{\B G_2\} & =  (\LL^7-1)^{-1}\Big (\LL^{-6}(\LL^2-1)^{-1}+(\LL-1) (\LL^6-1)^{-1}(\LL^2-1)^{-1}   \Big )  \\
& = (\LL^7-1)^{-1}\LL^{-6}(\LL^6-1)^{-1}(\LL^2-1)^{-1}\Big( \LL^6-1+(\LL-1)\LL^6 \Big)\\
& =\LL^{-6}(\LL^6-1)^{-1}(\LL^2-1)^{-1}.
\end{align*}
The last formula coincides with the formula for the class $\{G_2\}^{-1}$ (Proposition \ref{prop:formulas}).
\end{proof}

\subsection{The classes of $\B \Spin_7$ and $\B\Spin_8$}

Having computed the class of $\B G_2$, we are also ready to compute $\{\B \Spin_7\}$ and $\{\B\Spin_8\}$.

\begin{theorem}\label{prop:spin78}
We have $\{\B\Spin_7\}=\{\Spin_7\}^{-1}$ and $\{\B\Spin_8\}=\{\Spin_8\}^{-1}$ in $\kstack$.
\end{theorem}

\begin{proof}
Let us prove the result for $\Spin_7$ first. Consider the spin representation $V\cong k^8$ (see \cite[Section 9]{vistoli-spin}), where the basis has been chosen so that the action of $\Spin_7$ preserves the standard quadratic form $q(x_1,\hdots,x_7)=\sum_{i=1}^{7} x_i^2$.

We will stratify again the space $k^8$ as explained in (\ref{sec:setup}). Denote by $B_8, C_8$ and $Q_8$ the relevant loci in this case. As usual we have an isomorphism $[B_8/\Spin_7]\cong [(Q_8\times\GG_m)/(\Spin_7\times \mu_2)]$ where the action of $\mu_2$ is by scaling, and we obtain
\begin{equation}\label{eq:20}
\{\B\Spin_7\}(\LL^8-1)=\{[C_8/\Spin_7]\}+(\LL-1)\{[Q_8/(\Spin_7\times \mu_2)]\}.
\end{equation}
Now by {\cite[Proposition 4]{igusa}}, the action of $\Spin_7\times \mu_2$ on $Q_8$ is transitive, and the stabilizer is isomorphic to $G_2\times \mu_2$, {embedded in $\Spin7\times \mu_2$ via
$$
(g,1)\mapsto (g,1), \quad (g,-1)\mapsto ((-1)\cdot g, -1),
$$
where we are including $G_2\subset \Spin_7$ as the stabilizer of $e_1$, and the $(-1)$ in the last formula is the generator of the kernel of $\Spin_7\to \SO_7$.}
We deduce that
\begin{equation}\label{eq:21}
\{[Q_8/(\Spin_7\times \mu_2)]\}=\{\B(G_2\times\mu_2)\}=\LL^{-6}(\LL^6-1)^{-1}(\LL^2-1)^{-1} 
\end{equation}
where we used the fact that $\B(G_2\times \mu_2)\cong \B G_2\times \B\mu_2$, $\{\B\mu_2\}=1$ \cite[Proposition 3.2]{ekedahl-finite-group} and Theorem \ref{prop:g2}.

Now let us consider the first term. Again by {\cite[Proposition 4]{igusa}}, the action of $\Spin_7$ on $C_8$ is transitive, {and the stabilizer is a semidirect product $H\rtimes \SL_3$, where $H$ is a connected unipotent group of dimension $6$. By Proposition \ref{prop:unipotent} we have $\{\B H\}=\LL^{-6}$, and from $\{\GL_3\}=(\LL^3-1)(\LL^3-\LL)(\LL^3-\LL^2)$ and the fact that the determinant $\GL_3\to \bG_m$ is an $\SL_3$-torsor, we obtain
$$
\{\SL_3\}=\LL^3(\LL^3-1)(\LL^2-1).
$$
Moreover, the map $\B H\to \B(H\rtimes \SL_3)$ is also an $\SL_3$-torsor, and hence $\{\B H\}=\{\SL_3\}\cdot \{\B(H\rtimes \SL_3)\}$.
} Consequently for the {term $\{[C_8/\Spin_7]\}$} we obtain
$$
\{[C_8/\Spin_7]\}=\{\B(H\rtimes \SL_3)\}=\{\B H\}\cdot \{\SL_3\}^{-1}=\LL^{-9}(\LL^3-1)^{-1}(\LL^2-1)^{-1}.
$$
Combining the last formula with (\ref{eq:20}) and (\ref{eq:21}) gives
\begin{align*}
\{\B\Spin_7\} & =  (\LL^8-1)^{-1}\Big(\LL^{-9}(\LL^3-1)^{-1}(\LL^2-1)^{-1}+(\LL-1)\LL^{-6}(\LL^6-1)^{-1}(\LL^2-1)^{-1}   \Big)\\
&=  (\LL^8-1)^{-1}\LL^{-9}(\LL^2-1)^{-1}(\LL^6-1)^{-1}\Big( (\LL^3+1) + (\LL-1)\LL^3 \Big)\\
& =  \LL^{-9}(\LL^2-1)^{-1}(\LL^4-1)^{-1}(\LL^6-1)^{-1}.
\end{align*}
By comparing with the formula for $\{\Spin_7\}$ given by Proposition \ref{prop:formulas} we see that indeed $\{\B\Spin_7\}=\{\Spin_7\}^{-1}$.

Now let us turn to $\Spin_8$. Let us consider the linear representation on $V=k^8$ induced by the homomorphism $\Spin_8\to \SO_8$, and let us stratify the space $k^8$ as usual (we will study this situation for $\Spin_n$ with $n$ arbitrary in the next section). Let us denote again by $B_8, C_8$ and $Q_8$ the relevant strata. As in (\ref{sec:setup})
we have
\begin{equation}
\label{eq:10}
\{\B\Spin_8\}(\LL^8-1)=\{[C_8/\Spin_8]\}+(\LL-1)\{[Q_8/\Spin_8\times\mu_2]\}
\end{equation}
where $\mu_2$ acts by scaling. Moreover the action of $\Spin_8$ on $C_8$ and on $Q_8$ is transitive, {since the action of $\SO_8$ is}.

The stabilizer of the action of $\Spin_8\times\mu_2$ on $Q_8$ is isomorphic to $\Spin_7\times\mu_2$, {embedded via
$$
(g,1)\mapsto (g,1), \quad (g,-1)\mapsto (\eta \cdot g, -1),
$$
where we are including $\Spin_7\subset \Spin_8$ as the stabilizer of $e_1$, and $\eta = e_1\cdots e_8$ is the element of order $2$ of $\Spin_8$ that generates its center, together with $(-1)$.} Hence
\begin{equation}
\label{eq:11}
\{[Q_8/\Spin_8\times\mu_2]\}=\{\B(\Spin_7\times \mu_2)\}=\LL^{-9}(\LL^2-1)^{-1}(\LL^4-1)^{-1}(\LL^6-1)^{-1}.
\end{equation}
As for the first term, the stabilizer of $\Spin_8$ on $C_8$ is a semidirect product $\Spin_6\ltimes H$, where $H$ is a group scheme isomorphic to a $6$-dimensional vector space with addition (we will generalize this statement, in Proposition \ref{prop:nullcone} below). Using Proposition \ref{prop:linear} and Proposition \ref{prop:bg} (recall that $\Spin_6$ is special), we obtain
\begin{align*}
\{[C_8/\Spin_8]\} & =\{\B(\Spin_6\rtimes H)\} =  \LL^{-6}\{\B\Spin_6\} \\ & =\LL^{-12}(\LL^3-1)^{-1}(\LL^2-1)^{-1}(\LL^4-1)^{-1}
\end {align*}
We conclude the proof by plugging the last formula and (\ref{eq:11}) into (\ref{eq:10}), and comparing with Proposition \ref{prop:formulas}.
\end{proof}


\section{On the class of $\B\Spin_n$}

In this section we show that the computation of the class $\{\B\Spin_n\}$ in the Grothendieck ring of stacks $\kstack$ can be reduced to the computation of the classes $\{\B \Delta_n\}$ for various values of $n$, where $\Delta_n\subset \Pin_n$ is the {``extraspecial $2$-group'',} the finite subgroup of preimages of the diagonal matrices in $\O_n$ via the projection $\Pin_n\to \O_n$. Our main result here is that $\{\B\Spin_n\}=\{\Spin_n\}^{-1}$ is true for every $n\geq 2$ if and only if $\{\B \Delta_n\}=1$ for every $n$ (Corollary \ref{cor:main}). Note that for $n=1$ we have $\Spin_1\cong \mu_2$ (which is not connected), and $\{\B\Spin_1\}\cdot \{\Spin_1\}=2\neq 1$ in $\kstack$.

More precisely, because $\Spin_n$ is special for $2\leq n\leq 6$ and from the results of the previous section, we know that $\{\B\Spin_n\}=\{\Spin_n\}^{-1}$ holds for $2\leq n\leq 8$, and from this we can deduce that $\{\B \Delta_n\}=1$ for $n\leq 7$. Our result gives that, if $\{\B \Delta_n\}\neq 1$ for some $n$, and we let $n_0$ be the minimum positive integer for which this happens, then $\{\B\Spin_{n_0+1}\}\neq \{\Spin_{n_0+1}\}^{-1}$.

\begin{remark}
The fact that $\{\B G\}=1$ in $\kstack$ for a finite group $G$ might have some connection with stable rationality of quotients of faithful finite-dimensional representations of $G$, see the discussion in \cite[Section 6]{ekedahl-finite-group}. This problem for the {extraspecial $2$-group} $\Delta_n$ mentioned above is equivalent to the Noether problem for spin groups (as explained for example in \cite[Section 1.5.4]{bohing-rationality}), which is still open, and expected to have a negative answer for $\Spin_n$ with $n\geq 15$ (see \cite[Conjecture 4.5]{merk}).

A proof of the fact that $\{\B \Delta_n\}\neq 1$ for some $n$, although it would not directly imply a negative result of this kind, would certainly corroborate the expectation, and, should the relationship between the two phenomena be made precise, could possibly provide a line of attack for the rationality problem. 
\end{remark}

\subsection{Representations of $\Spin_n$ and $\Pin_n$ and stabilizers} In order to obtain information about the class $\{\B\Spin_n\}$, we will make use of the linear representations of $\Pin_n$ and $\Spin_n$ induced by the projections $\Pin_n\to \O_n$ and $\Spin_n\to \SO_n$ and the tautological representation of $\O_n$ and $\SO_n$ (for which we described the orbit structure and stabilizers in (\ref{sec:reps.on})).

This time assume that the quadratic form $q$ on $V=k^n$ is given by $$q(x_1,\cdots, x_n)=-(x_1^2+\cdots +x_n^2).$$ This quadratic form is more convenient if we want to do computations with (s)pin groups and the Clifford algebra. Since $k$ {has a square root of $-1$,} we can freely pass from between bases for which the quadratic form is given by the form $q_n$ in (\ref{sec:reps.on}) and for which it is given by this $q$. 

As in $(\ref{sec:reps.on})$, we will denote by $C$ the punctured null-cone $\{0\neq v\in V \mid q(v)=0\}$, by $B$ the complement $V\setminus \overline{C}$,  and by $Q$ the non-singular quadric $\{v\in V\mid q(v)=-1\}$ (the change of sign here will not be relevant). Assume that $n\geq 2$ if we are considering $\O_n$ and $n\geq 3$ if we are considering $\SO_n$, so that the actions are transitive on $C$.

\subsubsection{Stabilizers on the smooth quadric}\label{sec:smooth.stab}

Let us start by looking at the action of $\Spin_n$ on $V$. Since the action is via the projection $\rho_n\colon \Spin_n\to \SO_n$, the orbits are the same as for the action of $\SO_n$, and the stabilizers are preimages of stabilizers.

We will show later that the stabilizer of a vector in $C$ is isomorphic to a semidirect product $\Spin_{n-2}\ltimes W$, where $W=\langle e_1,e_2\rangle^\perp$ with its linear structure and $\Spin_{n-2}$ acts on $W$ via $\Spin_{n-2}\to \SO_{n-2}=\SO(W)$.

Let us consider now the action of $\Spin_n\times \mu_2$ on $V$, where $\mu_2$ acts by multiplication by $-1$, as usual.

\begin{proposition}
The stabilizer $G$ of $e_1\in Q$ for the action of $\Spin_n\times \mu_2$ on $V$ is isomorphic to $\Pin_{n-1}$.
\end{proposition}

\begin{remark}
In the proof of Theorem \ref{prop:spin78} we used the fact that this stabilizer for $n=8$ is also isomorphic to $\Spin_7\times \mu_2$. It is not true though that $\Pin_{n}\cong \Spin_{n}\times \mu_2$ for an arbitrary $n$, and this complicates the situation in the general case.
\end{remark}

\begin{proof}
The action of $\Spin_n\times \mu_2$ is via the map $\rho_n\times \id\colon \Spin_n\times \mu_2\to \SO_n\times \mu_2$, so the stabilizer is going to be the preimage of the stabilizer in $\SO_n\times \mu_2$, which is a always a copy of $\O_{n-1}$, but the embedding in $\SO_n\times \mu_2$ is different according to the parity of $n$.

If $n$ is even, then $\O_{n-1}\cong \SO_{n-1}\times \mu_2\subset \SO_n\times \mu_2$ embedded via $(M,\xi)\mapsto (\xi\cdot i(M),\xi)$, where $i\colon \SO_{n-1}\subset \SO_n$ is the inclusion as the stabilizer of $e_1$. Consequently, the stabilizer $G$ of $e_1$ in $\Spin_n\times \mu_2$ is the subgroup of of elements $(\alpha, \xi)$ such that $\rho_n(\alpha)=\xi\cdot i(M)$ for a (uniquely determined) element $M\in \SO_{n-1}$.

If $\xi=1$, this says that $\alpha\in \Spin_{n-1}$, embedded in $\Spin_n$ as the stabilizer of $e_1$. If $\xi=-1$, we get $\rho_n(\alpha)=-i(M)$, which can be written as $\sigma_1\circ i(-M)$, where $\sigma_1$ is the reflection across the hyperplane orthogonal to $e_1$ (i.e. the diagonal matrix with first entry $-1$, and all other entries equal to $1$). This implies that $\alpha=e_1\cdot \beta$, with $\beta \in \Pin_{n-1}\setminus \Spin_{n-1}$, since $\det(-M)=-\det(M)=-1$ because $n-1$ is odd. Here we are considering $\Pin_{n-1}\subset \Pin_n$ again as the stabilizer of $e_1$. 

Let us consider the function $\phi\colon \Pin_{n-1}\to G\subset \Spin_n\times \mu_2$ given by $\beta \mapsto ( \beta , 1)$ if $\det(\rho_{n-1}(\beta))=1$ (so that actually $\beta\in \Spin_{n-1}\subset \Pin_{n-1}$), and $\beta \mapsto (e_1\cdot \beta, -1)$ if $\det(\rho_{n-1}(\beta))=-1$. We claim that this is an isomorphism.

It is clearly injective and surjective, so we just have to check that it is a homomorphism. Let us consider $\beta, \beta'\in \Pin_{n-1}$, and let us compare $\phi(\beta)\phi(\beta')$ and $\phi(\beta\beta')$. The components in the $\mu_2$ factor are obviously equal, so let us only worry about the component in $\Spin_n$. We have $4$ cases to treat:
\begin{itemize}
\item $\det(\rho_{n-1}(\beta))=\det(\rho_{n-1}(\beta'))=1$: \\ we have $\phi(\beta)=\beta, \phi(\beta')=\beta'$ and $\phi(\beta\beta')=\beta\beta'$, so $\phi(\beta)\phi(\beta')=\phi(\beta\beta')$.

\item $\det(\rho_{n-1}(\beta))=-1$ and $\det(\rho_{n-1}(\beta'))=1$: \\ we have $\phi(\beta)=e_1\cdot \beta, \phi(\beta')=\beta'$ and $\phi(\beta\beta')=e_1\cdot \beta\beta'$, so $\phi(\beta)\phi(\beta')=\phi(\beta\beta')$.

\item $\det(\rho_{n-1}(\beta))=1$ and $\det(\rho_{n-1}(\beta'))=-1$: \\ in this case $\phi(\beta)=\beta, \phi(\beta')=e_1\cdot \beta'$, so $\phi(\beta)\phi(\beta')=\beta\cdot e_1\cdot \beta'$, and $\phi(\beta\beta')=e_1\cdot \beta\beta'$. Note that $\beta\cdot e_1=e_1\cdot \beta$ (because $\beta\in \Spin_{n-1}$ can be written in the Clifford algebra as sum of products of an even number of $e_i$ with $i\geq 2$, and each of these anticommutes with $e_1$), so we find again that $\phi(\beta)\phi(\beta')=\phi(\beta\beta')$.

\item $\det(\rho_{n-1}(\beta))=\det(\rho_{n-1}(\beta'))=-1$: \\ in this case $\phi(\beta)=e_1\cdot \beta, \phi(\beta')=e_1\cdot \beta'$, so $\phi(\beta)\phi(\beta')=e_1\cdot \beta\cdot e_1\cdot \beta'$, and  $\phi(\beta\beta')=\beta\beta'$. Since $\beta \in \Pin_{n-1}\setminus \Spin_{n-1}$, in this case $\beta \cdot e_1=-e_1\cdot \beta$, so $e_1\cdot \beta\cdot e_1\cdot \beta'=-e_1^2\cdot \beta\beta'=\beta\beta'$, and we find again that $\phi(\beta)\phi(\beta')=\phi(\beta\beta')$.
\end{itemize}
The proof for odd $n$ is completely analogous, so we omit it.
\end{proof}

Because of the preceding proposition, if we want to compute $\{\B\Spin_n\}$ we also have to analyze the action of $\Pin_n$ on $V$, via the projection $\rho_n\colon \Pin_n\to \O_n$. As for $\Spin_n$, we will prove in the next section that the stabilizer of a vector in $C$ is isomorphic to a semidirect product $\Pin_{n-2}\ltimes W$ with a vector space. 

The situation for the action of $\Pin_n\times \mu_2$ on the smooth quadric is more complicated. Let $G_{n,1}\subset \Pin_n\times \mu_2$ be the stabilizer of $e_1$. Note that the elements of $G_{n,1}$ are exactly the pairs $(\alpha, \xi)$ such that $\rho_n(\alpha)(e_1)=\pm e_1$, and $\xi$ is uniquely determined by the sign in the last formula (we can succinctly write $\rho_n(\alpha)(e_1)=\xi\cdot  e_1$). So the composite $G_{n,1}\to \Pin_n$ is injective, and we can think of $G_{n,1}$ as the subgroup of $\Pin_n$ that stabilizes the set $\{\pm e_1\}$.

The group $G_{n,1}$ contains the stabilizer of $e_1$ in $\Pin_n$, which is a copy of $\Pin_{n-1}$, and the quotient is a copy of $\mu_2$ (the map $G_{n,1}\to \mu_2$ is the projection to the second factor of $\Pin_n\times \mu_2$). So there is a short exact sequence 
$$
\xymatrix{
0\ar[r] & \Pin_{n-1} \ar[r] & G_{n,1} \ar[r] & \mu_2 \ar[r] & 0
}
$$
which in general does not split.

The group $G_{n,1}$ acts on the orthogonal $V'=\langle e_1\rangle^\perp\cong k^{n-1}$ via $\O_{n-1}$, and we can consider the orbits and stabilizers for this action. On the locus $C$, as for $\Spin_n$ and $\Pin_n$, we will prove that the stabilizer is isomorphic to a semidirect product $G_{n-2,1} \ltimes W$ with a vector space. As before, the stabilizer of the action of $G_{n,1}\times \mu_2\subset \Pin_n\times \mu_2^2$ on $Q$ is more complicated. Denote this stabilizer by $G_{n,2}$. This is the subgroup of triples $(\alpha, \xi_1,\xi_2)\in \Pin_n\times \mu_2^2$ such that $\rho_n(\alpha)(e_i)=\xi_i\cdot e_i$ for $i=1,2$, and can be identified by the subgroup of $\Pin_n$ that fixes the sets $\{\pm e_1\}$ and $\{\pm e_2\}$. This group is an extension
$$
\xymatrix{
0\ar[r] & \Pin_{n-2} \ar[r] & G_{n,2} \ar[r] & \mu_2^2 \ar[r] & 0.
}
$$
Let us iterate this: let $G_{n,r}\subset  \Pin_n\times \mu_2^r$ be the subgroup of elements $(\alpha,\xi_1,\hdots,\xi_r)$ such that $\rho_n(\alpha)(e_i)=\xi_i \cdot e_i$ for all $i=1,\hdots, r$. This can be seen as the subgroup of $\Pin_n$ of elements that fix the sets $\{\pm e_1\},\hdots, \{\pm e_r\}$, and it is an extension
$$
\xymatrix{
0\ar[r] & \Pin_{n-r} \ar[r] & G_{n,r} \ar[r] & \mu_2^r \ar[r] & 0.
}
$$
There is a natural surjective homomorphism $G_{n,r}\to \O_{n-r}$, and $G_{n,r}$ acts on $\langle e_1,\hdots, e_r\rangle^\perp\cong k^{n-r}$ via this map (preserving the induced quadratic form). Using the usual stratification for this action on $k^{n-r}$, we will show that the stabilizer of a point of $C$ is isomorphic to a semidirect product  $G_{n-2,r} \ltimes W$ with a vector space (note that we assumed $n\geq 2$, so $n-2\geq 0$, and we are implicitly assuming that $r\leq n-2$). The stabilizer of a point of $Q$ in $G_{n,r}\times \mu_2$ is a copy of $G_{n,r+1}$. Note that for $r=n-1$, we have $G_{n,n-1}=G_{n,n}$: indeed, an element of $\O_n$ that sends $e_i$ to $\pm e_i$ for $i=1,\hdots, n-1$ also has to send $e_n$ to $\pm e_n$.

For $r=n$ we obtain a finite group $\Delta_n=G_{n,n}$ of cardinality $2^{n+1}$, that is an extension
$$
\xymatrix{
0\ar[r] & \Pin_{0}\cong \mu_2 \ar[r] & \Delta_n \ar[r] & \mu_2^n \ar[r] & 0
}
$$
and that coincides with the preimage of the group of the diagonal matrices in $\O_n$ along the map $\rho_n\colon \Pin_n\to \O_n$. Note that $\Delta_n\subset \Pin_n$ is also isomorphic to the preimage of the diagonal matrices of $\SO_{n+1}$ in $\Spin_{n+1}$, by sending the element $e_i \in C_n$ to $e_ie_{n+1}\in C_{n+1}$ (recall that $C_n$ denotes the Clifford algebra). This finite subgroup is quite well-understood, and is responsible in particular for the exponential growth of the essential dimension of spin groups (see \cite{BRV}).

\subsubsection{Stabilizers on the punctured null-cone}\label{sec:stab.nullcone}

Now let us consider the stabilizer of a vector in the punctured null-cone $C\subset V$ for the action of $\Spin_n$, or $\Pin_n$, or one of the groups $G_{N,r}$ with $1\leq r\leq N$ introduced in the previous section. Since the proof will be the same in all cases, let us set $G_n$ to be one of these groups (and then $G_{n-2}$ will be the same type of group, with the index shifted down by $2$).

\begin{proposition}\label{prop:nullcone}
The stabilizer of a point in $C$ for the action of $G_n$ is isomorphic to a semidirect product $G_{n-2} \ltimes W$ with a vector space $W$ of dimension $n-2$ (that we see as an algebraic group via its linear structure), where the action of $G_{n-2}$ on $W$ is linear.
\end{proposition}

\begin{proof}
Let $H_n$ denote $\O_n$ if $G_n=\Pin_n$ or $G_{N,r}$ with $N-r=n$, and $\SO_n$ if $G_n=\Spin_n$. Recall that if $S'$ is the stabilizer of a vector in $C$ for the group $H_n$, let us take the specific vector to be $\frac{\sqrt{-1}e_1+e_2}{2}$, then there is a short exact sequence
$$
\xymatrix{
0\ar[r] & W \ar[r] & S'  \ar[r] & H_{n-2} \ar[r] & 0
}
$$
where we have explicit formulas for the linear transformation $\phi_w \in S'\subset H_n$ associated to a vector $w\in W$ (see the description in (\ref{sec:reps.on})).

The stabilizer $S$ of a vector in $C$ for the action of $G_n$ is the preimage of $S'$ via the surjective homomorphism $G_n\to H_n$. Note that there is an injective homomorphism $G_{n-2}\to S$, induced by the diagram \bigskip
$$
\xymatrix{
G_{n-2}\ar@/^1pc/[rr]\ar@{-->}[r]\ar[d]& S\ar[r]\ar[d] & G_n\ar[d]\\
H_{n-2}\ar[r] & S'\ar[r] & H_{n}
}
$$
where the rightmost square and the rectangle are cartesian.

We will prove that
\begin{itemize}
\item the map $W\to S'$ lifts to an (injective) homomorphism $W\to S$,
\item the images of $W$ and of $G_{n-2}$ intersect trivially in $S$, and
\item the image of $W$ in $S$ is a normal subgroup.
\end{itemize}
From this it will follow that we have a diagram with exact rows
$$
\xymatrix{
0\ar[r] & W \ar[r]\ar[d]^{=} & S \ar[d] \ar[r] & G_{n-2} \ar[r]\ar[d] & 0\\
0\ar[r] & W \ar[r] & S'  \ar[r] & H_{n-2} \ar[r] & 0
}
$$
where $S\to G_{n-2}$ has a section, and hence $S\cong G_{n-2}\ltimes W$, as desired. Moreover the action of $G_{n-2}$ on $W$ will be linear, since it is via $H_{n-2}$.

Let us turn to proving the three claims. For this, let $f_1=\frac{\sqrt{-1}e_1+e_2}{2}$ and $f_2=\frac{\sqrt{-1}e_1-e_2}{2}$. These are two isotropic vectors for the quadratic form $q$, and $h(f_1,f_2)=\frac{1}{2}$. Let us define a function $f\colon W\to C_n$ to the Clifford algebra of the quadratic form $q$, by $$f(w)=wf_1+1.$$ We will check that $f$ is a homomorphism to the group $S$, that lifts $W\to S'$.

First we note that $f(w)$ is an element of $\Spin_n\subset \Pin_n$ (possibly $\subset G_{N,r}$ with $n=N-r$): it is an even element, and it is a product of two vectors in $V$ of length $1$, i.e. such that $\|v\|^2=1$, where for $v\in V=k^n$ we write $$\|v\|^2=\|(x_1,\hdots,x_n)\|^2=x_1^2+\cdots x_n^2=-q(v).$$ Namely, if $\|w\|^2\neq 0$, then we can write
$$
f(w)=wf_1+1=\frac{w}{a}\cdot \frac{\|w\|^2f_1-w}{a}
$$
where $a\in k$ is such that $a^2=\|w\|^2$ (recall that in the Clifford algebra $C_n$ we have $v^2=q(v)=-\|v\|^2$ for $v\in V$). Both $\frac{w}{a}$ and $\frac{\|w\|^2f_1-w}{a}$ are vectors in $V$ of length $1$, so $f(w) \in \Spin_n$.

If $\|w\|^2=0$, a computation shows that
$$
f(w)=wf_1+1=\left(e_1-\frac{1}{2}\sqrt{-1}w\right)\cdot e_1 \cdot \left( e_2-\frac{1}{2}w\right)\cdot e_2
$$
in the Clifford algebra $C_n$, and all four terms are vectors of length $1$ in $V$, hence again $f(w)\in \Spin_n$.

Let us check that $f$ is a homomorphism: given $w,w'\in W$, we compute
$$
f(w)f(w')=(wf_1+1)(w'f_1+1)=wf_1w'f_1+(w+w')f_1+1=(w+w')f_1+1=f(w+w')
$$
since $wf_1w'f_1=-ww'(f_1)^2=-ww'q(f_1)=0$, because $f_1$ and $w$ are orthogonal and $f_1$ is isotropic.

Recall now that the image $\rho_n(\alpha)\in\O_n$ of an element $\alpha\in \Pin_n$ is the orthogonal transformation on $V$ given by $\rho_n(\alpha)v=\epsilon(\alpha)v\overline{\alpha}$, where $\epsilon(-)$ and $\overline{(-)}$ are two involutions of the Clifford algebra (see (\ref{sec:clifford})).

Straightforward computations in the Clifford algebra show that the map $f$ lands in the stabilizer of the vector $f_1$, and that the composite $W\to S\to S'$ coincides with the homomorphism $W\to S'$ given by $w\mapsto \phi_w$, where $\phi_w$ is the orthogonal transformation of $V$ described in (\ref{sec:reps.on}). It is also clear that the image of $f$ intersects $G_{n-2}\subset G_n$ trivial (recall that the products $e_{i_1}\cdots e_{i_k}$ are a basis of $C_n$ as a vector space).

The only claim left to prove is that the image of $f$ is a normal subgroup of $S$. Let us pick $\alpha\in S\subseteq G_n$, and check that $\alpha f(w) \alpha^{-1}$ is again in the image of $f$. Note that if $G_n=G_{N,r}$ with $N-r=n$, we can see $\alpha$ as an element of $\Pin_N$, whose associated element of $\O_N$ has a prescribed action on the first $r$ basis elements of $k^N$. In this case we can embed the Clifford algebra $C_n$ into $C_N$ in the natural way, and we will omit the subscript in $\Pin$ in order to treat all cases at once.

Furthermore, note that in the case of $G_{N,r}$ every element $\alpha$ can be written as a product $\alpha=\alpha' e_{i_1}\cdots e_{i_h}$ where $0 < i_1 <\hdots < i_h \leq r$ and $\alpha'\in \Pin_n$ (involving only the remaining generators). Conjugating an element of the form $f(w)$ by such an $\alpha$ gives the same result as conjugating by $\alpha'$, because $e_i (wf_1+1) e_i^{-1}=e_i (wf_1+1) (-e_i)=-e_iwf_1e_i+1=-e_i^2wf_1+1=wf_1+1$ since $e_i$ is orthogonal to both $f_1$ and $w$.

In any case, we know that since $\alpha\in \Pin$, in the corresponding Clifford algebra we have $\alpha\overline{\alpha}=1$, so $\alpha^{-1}=\overline{\alpha}$, and since $\alpha$ stabilizes $f_1$ we have $\epsilon(\alpha)f_1\overline{\alpha}=f_1$. By using the explicit formula for $f(w)$ we obtain $\alpha f(w) \alpha^{-1}=\alpha(wf_1+1)\overline{\alpha}=\alpha wf_1 \overline{\alpha} +1$.

Note that since $\alpha\overline{\alpha}=1$, we also have $\overline{\alpha}\alpha=1$ ($\overline{\alpha}$ is the inverse of $\alpha$ in $\Pin$, and right inverses in groups are also left inverses). By applying the involution $\epsilon$ to the last equality we obtain $\alpha^\mathrm{t}\epsilon(\alpha)=1$.

Consequently we have $$\alpha wf_1 \overline{\alpha}=\alpha w ( \alpha^\mathrm{t}\cdot \epsilon(\alpha))f_1\overline{\alpha}=(\alpha w \alpha^\mathrm{t})\cdot (\epsilon(\alpha)f_1\overline{\alpha})=(\alpha w \alpha^\mathrm{t})\cdot f_1.$$ Now note that $\epsilon(\alpha)(-w)\overline{\alpha}=w'$ is some vector in $V$ since $\alpha\in \Pin$, and
$$
-w'=\epsilon(w')=\epsilon(\epsilon(\alpha)(-w)\overline{\alpha})=\alpha w \alpha^\mathrm{t}
$$
hence $\alpha f(w) \alpha^{-1}=-w' f_1+1$ for some $w'\in V$. Moreover we know that the image $\rho_n(\alpha f(w) \alpha^{-1})= \rho_n(\alpha)\phi_w\rho_n(\alpha^{-1})\in S'$ is again in $W$, since $W\subset S'$ is a normal subgroup, hence is of the form $\phi_{w''}$ for some $w'' \in W$. We furthermore know that $f(w'')=w''f_1+1$ is one of the preimages of $\phi_{w''}$, and the other one has to be $-f(w'')=-w''f_1-1$ (since the kernel of $\Pin\to \O$ is $\{\pm 1\}$).

So either $-w'f_1+1=w''f_1+1$, in which case we are done (and in fact we will necessarily have $w''=w$), or $-w'f_1+1=-w''f_1-1$, which is easily checked to be impossible (write $-w'=ae_1+be_2+x$ with $a,b \in k$ and $x\in W$, and expand).
\end{proof}

\subsection{Putting everything together}

Let us make use of the computations of the previous sections, to obtain a formula for the class of $\B \Spin_n$ in terms of the subgroups that we introduced above. First we consider special cases for low values of $n$: we take $\Pin_0$ to be the trivial group, and we have $\Pin_1\cong \mu_4$, $\Spin_1\cong \mu_2$ and $\Spin_2\cong \gm$. These give $\{\B\Pin_0\}=\{\B\Pin_1\}=\{\B\Spin_1\}=1$ and $\{\B\Spin_2\}=(\LL-1)^{-1}$.

Now let us write down the formulas that we obtain by applying the procedure outlined in (\ref{sec:setup}), and by using the computation of the stabilizers that we carried out in (\ref{sec:smooth.stab}) and (\ref{sec:stab.nullcone}), together with Proposition \ref{prop:linear}.

Let us start with $\B\Spin_n$. We obtain
\begin{align*}
\{\B\Spin_n\}(\LL^n-1) & =\{\B(\Spin_{n-2}\ltimes k^{n-2})\} + (\LL-1)\{\B\Pin_{n-1}\}\\
& =\LL^{-n+2}\cdot \{\B\Spin_{n-2}\}+ (\LL-1)\{\B\Pin_{n-1}\}
\end{align*}
and by iterating, we see that we will be able to compute $\{\B\Spin_n\}$ if we can compute $\{\B\Pin_m\}$ for all $m<n$.

As for $\Pin_n$, we obtain
\begin{align*}
\{\B\Pin_n\}(\LL^n-1) & =\{\B(\Pin_{n-2}\ltimes k^{n-2})\} + (\LL-1)\{\B G_{n,1}\}\\
& =\LL^{-n+2}\cdot \{\B\Pin_{n-2}\}+ (\LL-1)\{\B G_{n,1}\}
\end{align*}
and for the group $G_{n,r}$ for $n\geq 2$, $n-2\geq r$ (so that also $r\leq n-2$),
\begin{align*}
\{\B G_{n,r}\}(\LL^{n-r}-1) & =\{\B(G_{n-2,r}\ltimes k^{n-r-2})\} + (\LL-1)\{\B G_{n,r+1}\}\\
& =\LL^{-n+r+2}\cdot \{\B G_{n-2,r}\}+ (\LL-1)\{\B G_{n,r+1}\}.
\end{align*}
In the special case of $G_{n,n-1}$, we already observed that this group is equal to $G_{n,n}=\Delta_n$, so that $\{\B G_{n,n-1}\}=\{\B G_{n,n}\}=\{\B \Delta_n\}$.

An easy induction argument using these formulas proves the following.

\begin{theorem}\label{thm:main}
The class $\{\B\Spin_n\}$ in $\kstack$ is a linear combination of the classes $\{\B \Delta_m\}$ with $m\leq n-1$, with coefficients in the subring $\Phi_\LL^{-1}\bZ[\LL]$.
Moreover, the coefficient of $\{\B \Delta_{n-1}\}$ in this linear combination is invertible in $\kstack$.
\end{theorem}

\begin{proof}
The first part of the statement is clear from the previous discussion. Let us comment on the second part.

From the recursive formulas the we obtained above, it is clear that in the expression for $\{\B\Spin_n\}$ the only term involving $\{\B \Delta_{n-1}\}$ will come from the term involving $\{\B\Pin_{n-1}\}$, and for each $G_{n,r}$ (including $r=0$), it will come from the term involving $\{\B G_{n,r+1}\}$.

By keeping track of the coefficient for the resulting term, we get
$$
\frac{\LL-1}{\LL^n-1}\cdot \frac{\LL-1}{\LL^{n-1}-1}\cdot \cdots \cdot \frac{\LL-1}{\LL^2-1}=\frac{(\LL-1)^{n-1}}{\prod_{i=2}^n (\LL^i-1)}
$$
which is invertible in $\kstack$.
\end{proof}

The theorem has the following consequences.

\begin{corollary}\label{cor:main}
The formula $\{\B\Spin_n\}=\{\Spin_n\}^{-1}$ holds for every $n\geq 2$ if and only if $\{\B \Delta_n\}=1$ for every $n$.
More precisely, assume that $\{\B \Delta_m\}=1$ for all $m<n-1$. Then $\{\B\Spin_n\}=\{\Spin_n\}^{-1}$ if and only if $\{\B \Delta_{n-1}\}=1$.
\end{corollary}

\begin{proof}
The second part is immediate from Theorem \ref{thm:main}: if $\{\B \Delta_m\}=1$ for every $m<n-1$, then $\{\B\Spin_n\}$ will be in the subring $\Phi_\LL^{-1}\bZ[\LL]$ of $\kstack$ if and only if $\{\B \Delta_{n-1}\}$ is (because the coefficients in front of these two terms in the equation that relates them are invertible elements of $\Phi_\LL^{-1}\bZ[\LL]$). Also, by Corollary \ref{cor:rat.function} we have $\{\B\Spin_n\} \in \Phi_\LL^{-1}\bZ[\LL]$ if and only if $\{\B\Spin_n\}=\{\Spin_n\}^{-1}$, and $\{\B \Delta_{n-1}\}\in \Phi_\LL^{-1}\bZ[\LL]$ if and only if $\{\B \Delta_{n-1}\}=1$.

For the first part, what we just proved shows that $\{\B \Delta_n\}=1$ for all $n$ implies that $\{\B\Spin_n\}=\{\Spin_n\}^{-1}$ for all $n$. Now assume $\{\B\Spin_n\}=\{\Spin_n\}^{-1}$ for all $n$, and that it is not the case that  $\{\B \Delta_n\}=1$ for all $n$. Take $n_0$ to be the minimum of the natural numbers $n\geq 1$ such that $\{\B \Delta_n\}\neq 1$. Then, because of the second part of the statement, we find $\{\B\Spin_{n_0+1}\}\neq \{\Spin_{n_0+1}\}^{-1}$, which is a contradiction.
\end{proof}

\begin{corollary}
We have that $\{\B \Delta_n\}=1$ for all $n\leq 7$.
\end{corollary}

\begin{proof}
This follows from the previous corollary, and the fact that $\{\B\Spin_n\}=\{\Spin_n\}^{-1}$ for $2\leq n\leq 8$: for $2\leq n\leq 6$ it is true because $\Spin_n$ is special (and by Proposition \ref{prop:bg}), and we have proved it in the previous section for $n=7,8$ (Theorem \ref{prop:spin78}).
\end{proof}

\begin{remark}
We stress that it not at all clear that $\{\B \Delta_n\}=1$ for these values of $n$ without resorting to this argument. Even for $n=2$ (for which we could invoke \cite[Proposition 2.3]{ivan}), trying to prove this directly turns out to be surprisingly complicated. This makes it somewhat unlikely that, with the current methods, it will be possible to check whether $\{\B\Spin_n\}$ is equal to $\{\Spin_n\}^{-1}$ or not by explicitly computing the class of the finite groups $\{\B \Delta_m\}$. It could be possible however to prove that this class is not $1$ indirectly, for example by showing that it has some non-trivial Ekedahl invariant (see \cite[Section 5]{ekedahl-finite-group}). We plan to return to this point in future work. 
\end{remark}

\bibliographystyle{alpha}
\bibliography{G2}

\end{document}